\documentclass[12pt,oneside,a4paper]{amsart}

\usepackage[british]{babel}

\usepackage{graphicx}
\usepackage{amscd}
\usepackage{amsmath}
\usepackage{amsfonts}
\usepackage{amssymb}
\usepackage{comment}
\usepackage{color}
\usepackage[usenames,dvipsnames,table,xcdraw]{xcolor}
\usepackage{pdfsync}
\usepackage{bbm}
\usepackage{dsfont}
\usepackage[colorlinks=true]{hyperref}
\usepackage{enumerate}
\usepackage{booktabs}
\usepackage{float}
\usepackage{wrapfig}
\usepackage{caption}
\usepackage{subcaption}
\usepackage{longtable}
\usepackage{listings}
\usepackage[T1]{fontenc}
\usepackage[scaled]{beramono}
\usepackage{tikz}
\usepackage{pgffor}
\usepackage[all]{xy}
\usepackage[mathscr]{euscript}

\newtheorem{theorem}{Theorem}[section]
\theoremstyle{plain}

\newtheorem{claim}[theorem]{Claim}
\newtheorem{fact}[theorem]{Fact}

\newtheorem{cor}[theorem]{Corollary}
\newtheorem{corollary}[theorem]{Corollary}

\newtheorem{definition}[theorem]{Definition}
\newtheorem{example}[theorem]{Example}

\newtheorem{lemma}[theorem]{Lemma}

\newtheorem{prop}[theorem]{Proposition}
\newtheorem{remark}[theorem]{Remark}

\numberwithin{equation}{section}

\newcommand{\R}{\mathbb{R}}
\newcommand{\N}{\mathbb{N}}

\newcommand{\ii}{\mathbf{i}}
\newcommand{\jj}{\mathbf{j}}

\newcommand{\iiv}{\overline{\imath}}
\newcommand{\jjv}{\overline{\jmath}}

\newcommand*{\e}[1]{\text{e}^{#1}}

\newcommand{\vertiii}[1]{{\left\vert\kern-0.25ex\left\vert\kern-0.25ex\left\vert #1
		\right\vert\kern-0.25ex\right\vert\kern-0.25ex\right\vert}}

\newcommand*{\arabicdec}[1]{\the\numexpr\value{#1}\relax}
\usepackage{anysize}
\usepackage{caption}

\definecolor{qqqqff}{rgb}{0,0,1}
\definecolor{ududff}{rgb}{0.30196078431372547,0.30196078431372547,1}
\definecolor{ccqqqq}{rgb}{0.8,0,0}
\definecolor{uuuuuu}{rgb}{0.26666666666666666,0.26666666666666666,0.26666666666666666}
\definecolor{cqcqcq}{rgb}{0.7529411764705882,0.7529411764705882,0.7529411764705882}

\begin{document}
\title[Hausdorff measure and Assouad dimension of generic conformal IFS]{ Hausdorff measure and Assouad dimension of generic self-conformal IFS on the line}

\author{Bal\'azs B\'ar\'any}
\address{Bal\'azs B\'ar\'any, Budapest University of Technology and Economics, MTA-BME Stochastics Research Group, P.O. Box 91, 1521 Budapest, Hungary} \email{balubs@math.bme.hu}

\author{Istv\'an Kolossv\'ary}
\address{Istv\'an Kolossv\'ary, University of St Andrews,  School of Mathematics and Statistics, St Andrews, KY16 9SS, Scotland;
\newline Budapest University of Technology and Economics, MTA-BME Stochastics Research Group, P.O. Box 91, 1521 Budapest, Hungary} \email{itk1@st-andrews.ac.uk}

\author{Micha\l\ Rams}
\address{Micha\l\ Rams, Institute of Mathematics, Polish Academy of Sciences,\newline  ul. \'Sniadeckich 8, 00-656 Warszawa, Poland}
\email{rams@impan.pl}

\author{K\'aroly Simon}
\address{K\'aroly Simon, Budapest University of Technology and Economics, MTA-BME Stochastics Research Group, P.O. Box 91, 1521 Budapest, Hungary} \email{simonk@math.bme.hu}


\begin{abstract}
This paper considers self-conformal iterated function systems (IFSs) on the real line whose first level cylinders overlap. In the space of self-conformal IFSs, we show that generically  (in topological sense) if the attractor of such a system has Hausdorff dimension less than $1$ then it has zero appropriate dimensional Hausdorff measure and its Assouad dimension is equal to $1$. Our main contribution is in showing that if the cylinders intersect then the IFS generically does not satisfy the weak separation property and hence, we may apply a recent result of Angelevska, K\"aenm\"aki and Troscheit. This phenomenon holds for transversal families (in particular for the translation family) typically, in the self-similar case, in both topological and in measure theoretical sense, and in the more general self-conformal case in the topological sense.

\end{abstract}
%\date{\today}

\maketitle

%\layout{}
\thispagestyle{empty}

%%%%%%%%%%%%%%%%%%%%%%%%%%%%%%%%%%%%%%%%%%%%%%%%%%%%%%%%%%%%%%%%%%%%%%%%%%%%%%%%
\section{Informal Statement}\label{sec:00}

A self-conformal iterated function system (IFS) on the real line is a finite collection $\mathcal{S}:=\left\{S_i\right\}_{i=1}^{m}$ of $\mathcal{C}^{1+\alpha}$ contracting conformal maps on a non-degenerate compact interval $X \subset \mathbb{R}$ such that each $S_i$ extends to an injective, contracting conformal mapping on an open set that contains $X$. Let us recall that a map $f\colon \R\mapsto\R$ is $\mathcal{C}^{1+\alpha}$ if $f$ is continuously differentiable and the derivative $f'$ is non-vanishing and H\"older continuous with H\"older exponent $\alpha$. It is well-known that there exists a unique, non-empty compact set, the attractor $\Lambda$ associated to $\mathcal{S}$, which satisfies
\begin{equation*}
\Lambda = \bigcup_{i=1}^m S_i(\Lambda).
\end{equation*}
The set $\Lambda$ is called a self-conformal set.

Let $\Sigma^*:=\bigcup_{n=1}^{\infty }\left\{1, \dots ,m\right\}^{n}$ be the collection of all finite length words $\iiv$, and we obtain $\iiv^-$ by dropping the last symbol of $\iiv$. For compositions of maps, we always write $S_{\iiv}=S_{i_1 \dots i_n}:=S_{i_1}\circ\cdots\circ S_{i_n}$. Numerous different separation conditions exist in the literature depending on the extent of separation between the cylinder sets $S_i(\Lambda)$. Of these, let us recall the following. An IFS $\mathcal{S}$ with attractor $\Lambda$
\begin{enumerate}[(i)]
\item has an \texttt{exact overlap} if
\begin{equation*}
\text{there exist } \iiv,\jjv\in\Sigma^\ast,\ \iiv\ne\jjv \mbox{ such that }
S_{\iiv}|_{\Lambda}\equiv S_{\jjv}|_{\Lambda};
\end{equation*}
\item satisfies the \texttt{Strong Separation Property (SSP)} if
\begin{equation*}
S_i(\Lambda)\cap S_j(\Lambda)=\emptyset \text{ for distinct } i,j\in\{1,\ldots,m\};
\end{equation*}
\item satisfies the \texttt{Weak Separation Property (WSP)} \cite{lau1999multifractal, Zerner_WSP_96ProcAMS} if
\begin{equation}\label{eq:92}
\sup\left\{
\#\Phi(x,r):
x\in\Lambda \mbox{ and }
r>0
\right\}<\infty,
\end{equation}
where $B(x,r)$ is the ball in $\mathbb{R}^d$  centered at $x$ with radius $r$ and where
\begin{equation}\label{eq:93}
\Phi(x,r):=
\big\{
S_{\iiv}|_\Lambda:
\mathrm{diam}(S_{\iiv}(\Lambda))
\leq
r
<
\mathrm{diam}(S_{\iiv^-}(\Lambda)),\,
%\\
S_{\iiv}(\Lambda)\cap B(x,r)\neq  \emptyset,\ \iiv\in\Sigma^{\ast}
\big\}.
\end{equation}
\end{enumerate}
The WSP is strictly weaker than the SSP, moreover, an IFS with an exact overlap can satisfy the WSP but never the SSP. We remark that the Open Set Condition, which is weaker than the SSP but stronger than the WSP, plays an important role in our proofs, but is not essential in stating our results. Therefore, we postpone its definition to Section~\ref{sec:98}.

If a conformal IFS on $\mathbb{R}$ satisfies the SSP,   then the Hausdorff and Assouad dimension of its attractor $\Lambda$, denoted $\dim_{\mathrm H}\Lambda$ and $\dim_{\mathrm A}\Lambda$, are both equal to the so-called conformal dimension, see Section~\ref{sec:98} for the definition. Furthermore, the $\dim_{\mathrm H}\Lambda$-dimensional Hausdorff measure $\mathcal{H}^{\dim_{\mathrm H}\Lambda}(\Lambda)\in(0,\infty)$. In fact, recently Angelevska, K\"aenm\"aki and Troscheit showed in~\cite{KaenmakiEtal_SelfConf_20BLMS} that this holds true even under the weaker WSP condition, provided that $\dim_{\mathrm H}\Lambda<1$. Moreover, failure of the WSP implies that $\mathcal{H}^{\dim_{\mathrm H}\Lambda}(\Lambda)=0$ and $\dim_{\mathrm A}\Lambda=1$.

Notice that the SSP is an open condition, i.e. if an IFS satisfies the SSP, then a small enough perturbation of it will still satisfy the SSP. The main question this paper addresses is if a conformal IFS on $\mathbb{R}$ does {\bf not} satisfy the SSP, then in some generic sense, how significant are the overlaps between the cylinder sets?
\vspace{1mm}

\noindent{\bf Informal statement.} Our main result is that on a proper space of conformal IFSs on $\mathbb{R}$, it is a generic property that if an IFS does not satisfy the SSP, then
%both in a topological and a measure theoretical sense,
\begin{equation*}
\text{it does {\bf not} satisfy the WSP {\bf nor} does it have exact overlaps.}
\end{equation*}
In particular, combining this with the aforementioned characterization of the WSP in~\cite{KaenmakiEtal_SelfConf_20BLMS}, we get that restricting to conformal IFSs with conformal dimension $<1$,
\begin{equation*}
\text{generically, failure of SSP} \;\;\Longrightarrow\;\; \mathcal{H}^{\dim_{\mathrm H}\Lambda}(\Lambda)=0 \;\text{ and }\; \dim_{\mathrm A}\Lambda=1.
\end{equation*}
In the next section, we define all the necessary terminology in order to state our results precisely.

%%%%%%%%%%%%%%%%%%%%%%%%%%%%%%%%%%%%%%%%%%%%%%%%%%%%%%%%%%%%%%%%%%%%%%%%%%%%%%%%
\section{Further motivation and main results}\label{sec:01}

A self-similar IFS $\mathcal{S}$ on $\mathbb{R}$ is a special conformal IFS consisting of strictly contracting similarity maps:
\begin{equation}\label{eq:19}
  \mathcal{S}:=
  \left\{
  S_i(x)= r_i x +t_i
  \right\}_{i=1}^{m}, x\in\mathbb{R}, r_i\in (-1,1)\setminus\left\{0\right\}, t_i\in \mathbb{R}.
\end{equation}
The attractor $\Lambda$ of $\mathcal{S}$ is called a self-similar set.

The  Assouad dimension  was introduced by  P. Assouad \cite{assouad1977espaces},\cite{assouad1979etude} in relation with quasi-conformal mappings and embeddability problems, see
\cite{heinonen2012lectures, luukkainen1998assouad, robinson2010dimensions}. Recently though, considerable attention has been given to its study in Fractal geometry, see
\cite{KaenmakiEtal_SelfConf_20BLMS, falconer2019assouad, FarkasFraser_AssouadSelfSim_2015AdvMath, FraserEtal_AssouadSelfSim_2015AdvMath, garca2019assouad, kaenmaki2015weak, orponen2019assouad} or the recent book~\cite{fraser2020assouadBook} to name a few. The Assouad dimension of a set $E \subset \mathbb{R}^d$ is
\begin{multline*}
\dim_{\rm A}E :=\inf\Bigg\{\alpha:\; \exists \ C>0,  \mbox{ such that }\, \forall\  0<r<R,  \\
\sup\limits_{x\in E}
N_r\left(B(x,R)\cap E\right) \leq C\left(\frac{R}{r}\right)^{\alpha}
\Bigg\},
\end{multline*}
where $N_r\left(F\right)$
is the smallest number of open balls centered in the set $F \subset \mathbb{R}^d$ of radius $r$
with which we can cover $F$. We denote the Hausdorff dimension of a set $E \subset \mathbb{R}^d$ by $\dim_{\mathrm H}E$ and its $s$-dimensional Hausdorff measure by $\mathcal{H}^s(E)$.

In particular, on $\mathbb{R}$, an interesting dichotomy was proved between the separation property WSP and the Assouad dimension of self-similar sets by Fraser, Henderson, Olson and Robinson in~\cite[Theorem~1.3.]{FraserEtal_AssouadSelfSim_2015AdvMath} and extended to self-conformal sets by Angelevska, K\"aenm\"aki and Troscheit in \cite[Theorem~4.1, Corollary~4.2]{KaenmakiEtal_SelfConf_20BLMS}. We summarize these results in the following theorem.
\begin{theorem}[\cite{KaenmakiEtal_SelfConf_20BLMS, FraserEtal_AssouadSelfSim_2015AdvMath}]
 \label{thm:97}
  Assume that $\Lambda$ is the attractor of either a self-similar or  self-conformal IFS $\mathcal{S}$ on $\mathbb{R}$ such that $\Lambda$ is not a singleton. Then
 \begin{equation*}
\dim_{\mathrm A}\Lambda=\begin{cases}
 \dim_{\mathrm H}\Lambda, & \text{if }\mathcal{S} \text{ satisfies the Weak Separation Property},  \\
1, &\text{otherwise}.
\end{cases}
\end{equation*}
\end{theorem}
For a higher dimensional generalization see \cite{garca2019assouad}. For self-similar sets, Farkas and Fraser \cite[Corollary~3.2]{FarkasFraser_AssouadSelfSim_2015AdvMath} pointed out another equivalent characterization of the weak separation property, relating it to the positivity of the appropriate dimensional Hausdorff measure of the attractor. This result was also extended by Angelevska, K\"aenm\"aki and Troscheit in \cite[Corollary~4.2]{KaenmakiEtal_SelfConf_20BLMS} for self-conformal sets.

\begin{theorem}[\cite{KaenmakiEtal_SelfConf_20BLMS, FarkasFraser_AssouadSelfSim_2015AdvMath}]
 \label{thm:hausmeas}
  Assume that $\Lambda$ is the attractor of either a self-similar or  self-conformal IFS $\mathcal{S}$ on $\mathbb{R}$ such that $\Lambda$ is not a singleton and $\dim_{\rm H}\Lambda<1$. Then
\begin{align*}
0<&\mathcal{H}^{\dim_{\rm H}\Lambda}(\Lambda)<\infty,  &&\text{if }\mathcal{S} \text{ satisfies the Weak Separation Property}, \\
&\mathcal{H}^{\dim_{\rm H}\Lambda}(\Lambda)=0, &&\text{otherwise}.
\end{align*}
\end{theorem}

We are now ready to introduce the space of self-conformal IFSs in Section~\ref{sec:11}, moreover, parameterized families of self-conformal IFSs satisfying a transversality condition in Section~\ref{sec:12}.

%%%%%%%%%%%%%%%%%%%%%%%%%%%%%%%%%%%%%%%%%%%%%%%%%%%%%%%%%%%%%%%%%%%%%%%%%%%
%%%%%%%%%%%%%%%%%%%%%%%%%%%%%%%%%%%%%%%%%%%%%%%%%%%%%%%%%%%%%%%%%%%%%%%%%%%
\subsection{Generic self-conformal IFSs on the line}\label{sec:11}
We begin by defining the space of self-conformal IFSs on the line that we work with.

\begin{definition}[Conformal IFSs on the line]\label{def:96}
	Let $V \subset \mathbb{R}$ be a non-empty open interval.
	We say that a $\mathcal{C}^{1+\alpha}$ function
	$h:V\to\mathbb{R}$ is a conformal mapping if it has non-vanishing derivative.
	
	Furthermore, $\mathcal{S}=\left\{S_1, \dots ,S_m\right\}$ is a
	conformal IFS on a compact interval $X \subset \mathbb{R}$
	if $S_i:X\to X$ and $S_i$ extends to a conformal injective mapping
	$S_i:V\to V$, where $V$ is an open interval with $X \subset V$ and
	\begin{equation*}
	\sup\limits_{x\in V}|S'_i(x)|<1 \;\;\text{for all } i\in[m]:=\{1,\ldots,m\}.
	\end{equation*}
\end{definition}

\begin{definition}[Space of self-conformal IFSs on the line]\label{def:95}
Let $\Theta^m(X)$ denote the collection of conformal IFSs of $m$ mappings on the compact interval $X$. For $0<\beta<\rho<1$, we define
	\begin{equation*}
	\Theta_{\beta,\rho}^m(X):
	=
	\big\{\mathcal{S}\in \Theta^m(X):
	\beta \leq |S'_i(x)| \leq \rho,\ \forall x\in X,\ i\in[m]
	\big\}.
	\end{equation*}
	
	For a function $h\in \mathcal{C}^{1+\alpha}(X,\mathbb{R})$  we write
	$$
	\vertiii{h}:=
	\|h\|_{\sup}
	+
	\|h'\|_{\sup}
	+
	\sup\limits_{x,y\in X}\frac{|h'(x)-h'(y)|}{|x-y|^{\alpha}}.
	$$
	For $\mathcal{T}=(T_1, \dots ,T_m),\mathcal{G}=(G_1, \dots ,G_m)\in\Theta_{\beta,\rho}^m(X)$ we define
	\begin{equation}\label{eq:38}
	\mathrm{dist}(\mathcal{T},\mathcal{G}):=
	\max\limits_{i\in[m]}
	\vertiii{T_i-G_i}.
	\end{equation}
	Then $\Theta_{\beta,\rho}^m(X)$ endowed with this metric is a complete metric space.
\end{definition}

We remind the reader that a subset of a topological space is called \texttt{a set of first category} or \texttt{meagre} if it is a countable union of nowhere dense closed sets (i.e. the interior is an empty set). We call a subset of a topological space a \texttt{$G_\delta$-set} (or have the $G_\delta$ property)  if it is a countable intersection of open sets. Observe that the complement of a set of first category is a dense $G_\delta$-set by definition.
We can now define a `generic' self-conformal IFS on $\Theta_{\beta,\rho}^m(X)$.

\begin{definition}\label{ter:99}
	Let $\mathcal{P}$ be a property of self-conformal IFSs on the line.
	We say that \texttt{a generic self-conformal IFS on the line has property $\mathcal{P}$} if for all non-empty compact intervals
	$X \subset \mathbb{R}$, $m>2$ and $0<\beta<\rho<1$ the set of IFSs from $\Theta_{\beta,\rho}^m(X)$ which do \emph{not}
	have property $\mathcal{P}$ is a set of first category.
\end{definition}

Our main result is the following. Its proof is in Section~\ref{sec:04}.

\begin{theorem}\label{thm:98}
	For a generic (in the sense of Definition~\ref{ter:99}) self-conformal IFS on the line either
	\begin{equation*}
	\text{the SSP holds } \;\text{ \emph{OR} }\; \text{ the WSP does not hold and there are no exact overlaps.}
	\end{equation*}
\end{theorem}

Actually, we prove a stronger result on the genericity of the failure of the WSP, see Section~\ref{sec:04} for the precise details.

The conformal dimension, see Section~\ref{sec:98} for the definition, is a natural upper bound for $\dim_{\mathrm H}\Lambda$. Combining Theorem~\ref{thm:98} with \cite[Theorem~4.1, Corollary~4.2]{KaenmakiEtal_SelfConf_20BLMS} immediately gives the following.
\begin{cor}\label{cor:01}
	Let us restrict to the open set of self-conformal IFSs on the line with conformal dimension strictly less than $1$. Then, a generic IFS either
	\begin{equation*}
	\text{satisfies the SSP} \;\;\text{ \emph{OR} }\;\; \mathcal{H}^{\dim_{\mathrm H}\Lambda}(\Lambda)=0 \,\text{ and }\, \dim_{\mathrm A}\Lambda=1.
	\end{equation*}
\end{cor}

\begin{remark}
Even though the set of IFSs for which the WSP fails and there are no exact overlaps is topologically large, it is not easy to constract a concrete example of such an IFS.
\end{remark}

%%%%%%%%%%%%%%%%%%%%%%%%%%%%%%%%%%%%%%%%%%%%%%%%%%%%%%%%%%%%%%%%%%%%%%%%%%%
%%%%%%%%%%%%%%%%%%%%%%%%%%%%%%%%%%%%%%%%%%%%%%%%%%%%%%%%%%%%%%%%%%%%%%%%%%%
\subsection{Transversal families of conformal IFSs on the line}\label{sec:12}

Let $d \geq 1$ be an integer and $B \subset \mathbb{R}^d$ be a non-degenerate compact ball. For every $\underline{\lambda} \in B$ we are given an IFS $\mathcal{S}^{\underline{\lambda} }:=
\big\{S_{1}^{\underline{\lambda} }, \dots ,
S_{m}^{\underline{\lambda} }\big\}$ on $\mathbb{R}$. To show dependence on $\underline{\lambda}$, we denote the attractor of $\mathcal{S}^{\underline{\lambda}}$ by $\Lambda^{\underline{\lambda}}$.

\begin{definition}\label{def:97}
We say that $\big\{ \mathcal{S}^{\underline{\lambda} }\big\}_{\underline{\lambda} \in B} $ is a \texttt{family of  self-conformal IFSs on the line} if the following two conditions hold:
	\begin{description}
		\item[(a)] for every $\underline{\lambda} \in B$ we have $\mathcal{S}^{\underline{\lambda} }\in \Theta_{\beta,\rho}^m(X)$;
		\item[(b)] for every $\mathbf{i}\in\Sigma=\{1,\ldots,m\}^{\mathbb{N}}$ the following mapping is $\mathcal{C}^1:$
		\begin{equation}\label{eq:42}
		\underline{\lambda} \mapsto
		\Pi_{\underline{\lambda} }(\mathbf{i}):= \lim\limits_{n\to\infty} S_{i_1 \dots i_n}^{\underline{\lambda}}(x),\quad\underline{\lambda}\in B,
		\end{equation}
		where $x\in X$ is arbitrary.
	\end{description}
In particular, if $\mathcal{S}^{\underline{\lambda}}$ is a self-similar IFS for every $\underline{\lambda} \in B$, recall~\eqref{eq:19}, then we call it a \texttt{family of  self-similar IFSs on the line}.
\end{definition}

\begin{example}  [Translation family]\label{ex:98}
	Let $\mathcal{S}\in \Theta_{\beta,\rho}^m(X)$ and let $B \subset \mathbb{R}^m$ be a non-degenerate compact ball. Then the translation family of $\mathcal{S}$ is
	\begin{equation*}
	\mathcal{S}^{\underline{\lambda}} :=
	\left\{S_i(x)+\lambda_i\right\}_{i=1}^{m},\quad
	\underline{\lambda}=\{\lambda_1,\ldots,\lambda_m\} \in B.
	\end{equation*}
\end{example}

\begin{definition}[Transversal family of self-conformal IFSs on the line]\label{def:98}
	We say that a family of self-conformal IFSs $\{\mathcal{S}^{\underline{\lambda} } \}_{\underline{\lambda} \in B}$
	(in the sense of Definition~\ref{def:97})
	is a transversal family if the following \emph{transversality condition}
	holds:
	
	There exists an $\zeta>0$ such that for all $\mathbf{i},\mathbf{j}\in \Sigma$ with $i_1\ne j_1$,
	\begin{equation}\label{eq:15}
	\left|
	\Pi_{\underline{\lambda} }(\mathbf{i})-\Pi_{\underline{\lambda} }(\mathbf{j})
	\right|<\zeta \;\Longrightarrow\;
	\left\|
	\nabla_{\underline{\lambda} }\left(
	\Pi_{\underline{\lambda} }(\mathbf{i})-\Pi_{\underline{\lambda} }(\mathbf{j})
	\right)
	\right\|>\zeta,
	\end{equation}
	where $\nabla_{\underline{\lambda} }$ is the gradient in $\underline{\lambda} $.
\end{definition}

Let us introduce
\begin{align*}
\mathscr{SSP}&:= \{\underline{\lambda}\in B:\, \mathcal{S}^{\underline{\lambda}} \text{ satisfies the SSP}\}, \\
\mathscr{WSP}&:= \{\underline{\lambda}\in B:\, \mathcal{S}^{\underline{\lambda}} \text{ satisfies the WSP}\}, \\
\mathscr{EO}&:= \{\underline{\lambda}\in B:\, \mathcal{S}^{\underline{\lambda}} \text{ has an exact overlap}\}.
\end{align*}
The compliment of $\mathscr{SSP}$ is compact, because it is intersected with the compact set $B$ and $\mathscr{SSP}$ is open. Notice that $\mathscr{SSP}\cap \mathscr{EO}=\emptyset$. Let $\mathcal{L}_d(H)$ denote the $d$-dimensional Lebesgue measure of the subset $H\subset \mathbb{R}^d$. Our main result concerns the set of parameters $\big( (B\setminus \mathscr{SSP}) \cap \mathscr{WSP}\big) \cup \mathscr{EO}$, that is those parameters for which either we have exact overlap or we have overlaps and WSP.
We we prove that this set of parameters is a set of first category in the complete metric space $B\setminus \mathscr{SSP}$, which is the set of those parameters for which there are overlaps. Its proof is provided in  Section~\ref{sec:05}.

%We also use the following terminology. For a subset $H \subset \mathbb{R}^d$, we say that $H$ is a \texttt{very small set} if
%\begin{equation*}\label{eq:52}
%\mathcal{L}_d(H)=0 \;\mbox{ and }\; H
%\mbox{ is a set of first category},
%\end{equation*}
%where $\mathcal{L}_d(H)$ denotes the $d$-dimensional Lebesgue measure of $H$.

\begin{theorem}\label{thm:99}
Let $\{\mathcal{S}^{\underline{\lambda} }\}_{\underline{\lambda} \in B}$ be a transversal family of self-conformal IFSs on the line as in Definition \ref{def:98}. Then
\begin{equation*}
\big( (B\setminus \mathscr{SSP}) \cap \mathscr{WSP}\big) \cup \mathscr{EO} \mbox{ is a set of first category, moreover, } \mathcal{L}_d(\mathscr{EO})=0.
\end{equation*}
In addition, if $\{\mathcal{S}^{\underline{\lambda} }\}_{\underline{\lambda} \in B}$ is a self-similar family, then also $\mathcal{L}_d((B\setminus \mathscr{SSP}) \cap \mathscr{WSP})=0$.

\end{theorem}

\begin{remark}\
 The $G_\delta$ property does not follow from the more general Theorem~\ref{thm:98}, because it can be checked that $\{\mathcal{S}^{\underline{\lambda}}:\, \underline{\lambda} \in B\}$ as a subset of $\Theta_{\beta,\rho}^m$ is nowhere dense.

\end{remark}
%%%%%%%%%%%%%%%%%%%%%%%%%%%%%%%%%%%%%%%%%%%%%%%%%%%%%%%%%%%%%%%%%%%%%%%%%%%
%%%%%%%%%%%%%%%%%%%%%%%%%%%%%%%%%%%%%%%%%%%%%%%%%%%%%%%%%%%%%%%%%%%%%%%%%%%
\subsection{Application to translation families}\label{sec:96}

Recall the translation family in Example~\ref{ex:98}. The following lemma gives a sufficient condition for a translation family to satisfy the transversality condition~\eqref{eq:15}. We are unaware of a reference for it, so we provide the short proof.
\begin{lemma}\label{lemma:01}
Let $\{\mathcal{S}^{\underline{\lambda} }\}_{\underline{\lambda} \in B}$ be a translation family of the self-conformal IFS $\mathcal{S}\in \Theta_{\beta,\rho}^m(X)$, moreover, denote $\rho_i^{\ast}:= \sup_{x\in X} |S'_{i}(x)|$. If
\begin{equation}\label{eq:20}
\rho_i^{\ast} + \rho_j^{\ast} < 1 \quad \text{for all } i\neq j,
\end{equation}
then for every $\ii,\jj\in\Sigma$ such that $i_1\neq j_1$ there exists $p\in\{i_1,j_1\}$ for which
\begin{equation*}
\Big| \frac{\partial}{\partial \lambda_{p}} \big(\Pi_{\underline{\lambda} }(\mathbf{i})-\Pi_{\underline{\lambda} }(\mathbf{j})\big) \Big| \geq 1- \rho_{i_1}^{\ast} - \rho_{j_1}^{\ast} >0,
\end{equation*}
recall \eqref{eq:42}. Thus, $\{\mathcal{S}^{\underline{\lambda} }\}_{\underline{\lambda} \in B}$ is a transversal family and Theorem~\ref{thm:99} applies.
\end{lemma}
\begin{proof}
Without loss of generality we may assume that $\ii,\jj\in\Sigma$ such that $i_1=1$ and $j_1=2$. Let $\sigma: \Sigma\to\Sigma$ denote the left shift operator, $\delta_{k,\ell}$ is the Dirac delta and for fixed $\ii\in\Sigma$ and $z\in\{1,\ldots,m\}$ we write
\begin{equation*}
\mathcal{I}_z(\ii):=\{\ell>1:\; i_{\ell}=z\}.
\end{equation*}
Using \eqref{eq:42}, the partial derivative $\frac{\partial }{\partial \lambda_{z}}\Pi_{\underline{\lambda} }(\mathbf{i})$ is equal to
\begin{equation*}
\delta_{i_1,z} + S'_{i_1}(\Pi_{\underline{\lambda} }(\sigma\mathbf{i})) \Big( \delta_{i_2,z} + S'_{i_2}(\Pi_{\underline{\lambda} }(\sigma^2\mathbf{i}))\big( \delta_{i_3,z}+\cdots \big) \Big) = \delta_{i_1,z} +\sum_{\ell\in\mathcal{I}_z(\ii)} \prod_{k=1}^{\ell-1} S'_{i_k}(\Pi_{\underline{\lambda} }(\sigma^k\mathbf{i})).
\end{equation*}
For brevity, let $\rho_{\ii|\ell-1}:= \prod_{k=1}^{\ell-1} S'_{i_k}(\Pi_{\underline{\lambda} }(\sigma^k\mathbf{i})).$ Then, we can write
\begin{equation}\label{eq:11}
\frac{\partial}{\partial \lambda_{1}} \big(\Pi_{\underline{\lambda} }(\mathbf{i})-\Pi_{\underline{\lambda} }(\mathbf{j})\big) = 1 + E_1(\ii,\jj) \;\text{ and }\; \frac{\partial}{\partial \lambda_{2}} \big(\Pi_{\underline{\lambda} }(\mathbf{i})-\Pi_{\underline{\lambda} }(\mathbf{j})\big) = -1 + E_2(\ii,\jj),
\end{equation}
where
\begin{equation*}
E_z=E_z(\ii,\jj) = \sum_{\ell\in\mathcal{I}_z(\ii)} \rho_{\ii|\ell-1} - \sum_{\ell\in\mathcal{I}_z(\jj)} \rho_{\jj|\ell-1} \text{ for } z=1,2.
\end{equation*}
We take a convex linear combination of $|E_1|$ and $|E_2|$, for which we claim that
\begin{equation}\label{eq01}
\frac{(1-\rho_1^{\ast})|E_1|}{2-\rho_1^{\ast}-\rho_2^{\ast}} + \frac{(1-\rho_2^{\ast})|E_2|}{2-\rho_1^{\ast}-\rho_2^{\ast}} < (1-\rho_1^{\ast})|E_1|+(1-\rho_2^{\ast})|E_2| \leq \rho_1^{\ast}+\rho_2^{\ast} < 1.
\end{equation}
Indeed, the first and the last inequalities hold due to assumption~\eqref{eq:20}. The second inequality follows from union bounds
\begin{align*}
(1-\rho_1^{\ast})|E_1|+(1-\rho_2^{\ast})|E_2| &\leq  \Bigg| \sum_{\ell\in\mathcal{I}_1(\ii)} (1-\rho_{i_{\ell}}^{\ast})\rho_{\ii|\ell-1} \Bigg| + \Bigg| \sum_{\ell\in\mathcal{I}_1(\jj)} (1-\rho_{j_{\ell}}^{\ast})\rho_{\jj|\ell-1} \Bigg| \\
&+ \Bigg| \sum_{\ell\in\mathcal{I}_2(\ii)} (1-\rho_{i_{\ell}}^{\ast})\rho_{\ii|\ell-1} \Bigg| + \Bigg| \sum_{\ell\in\mathcal{I}_2(\jj)} (1-\rho_{j_{\ell}}^{\ast})\rho_{\jj|\ell-1} \Bigg| \\
&\leq \sum_{\ell=2}^{\infty} (1-\rho_{i_{\ell}}^{\ast})\rho_{i_1}^{\ast}\ldots \rho_{i_{\ell-1}}^{\ast} + \sum_{\ell=2}^{\infty} (1-\rho_{j_{\ell}}^{\ast})\rho_{j_1}^{\ast}\ldots \rho_{j_{\ell-1}}^{\ast} \\
&= \rho_{i_1}^{\ast} + \rho_{j_1}^{\ast} = \rho_{1}^{\ast} + \rho_{2}^{\ast}. \qquad(\text{telescopic sum})
\end{align*}

Choose $p\in\{1,2\}$ for which $|E_p|=\min\{|E_1|,|E_2|\}$. From the convex linear combination \eqref{eq01} it follows that $|E_p|\leq\rho_1^{\ast}+\rho_2^{\ast}<1$. This and~\eqref{eq:11} implies that for this choice of $p$ the assertion holds.
\end{proof}

In the self-similar setting, recall~\eqref{eq:19}, the translation family of a self-similar IFS $\mathcal{S}$ is of the form
\begin{equation}\label{eq:21}
\mathcal{S}^{\underline{\lambda}} :=
\left\{r_ix+t_i+\lambda_i\right\}_{i=1}^{m},\quad
\underline{\lambda} \in B\subset \mathbb{R}^m.
\end{equation}
The similarity dimension of $\mathcal{S}$ is defined as the unique solution of the equation $|r_{1}|^{s_0}+\cdots+|r_{m}|^{s_0}=1$. The condition in Lemma~\ref{lemma:01} simply becomes
\begin{equation*}
\max\limits_{i\ne j} \left\{|r_i|+|r_j|\right\}<1.
\end{equation*}
Combining this with the result of Simon and Solomyak~\cite[Theorem 2.1.]{SimonSolomyak_2002Fractals}, we get the following characterization. For simplicity, we refer to a set $H \subset \mathbb{R}^d$ as a \texttt{very small set} if it is of first category and $\mathcal{L}_d(H)=0$.

\begin{corollary}\label{cor:99}
Let $\{\mathcal{S}^{\underline{\lambda}}\}_{\underline{\lambda}\in B}$ be a translation family \eqref{eq:21} of the self-similar IFS $\mathcal{S}\in \Theta_{\beta,\rho}^m(X)$ with similarity dimension $s_0$.
Then apart from a very small set of $\underline{\lambda}$
\begin{enumerate}
\item if $s_0>1$, then $\mathcal{L}(\Lambda^{\underline{\lambda}})>0$. In particular, $\dim_{\mathrm H}\Lambda^{\underline{\lambda}}=\dim_{\mathrm A}\Lambda^{\underline{\lambda}}=1$;
\item if $s_0\leq 1$, then
\begin{enumerate}
\item either $\underline{\lambda}\in \mathscr{SSP}$, thus $\dim_{\mathrm H}\Lambda^{\underline{\lambda}}=\dim_{\mathrm A}\Lambda^{\underline{\lambda}}=s_0$ and $\mathcal{H}^{s_0}(\Lambda^{\underline{\lambda}})>0$;
\item or $\mathcal{H}^{\dim_{\mathrm H}\Lambda^{\underline{\lambda}}}(\Lambda^{\underline{\lambda}})=0$ and $\dim_{\mathrm A}\Lambda^{\underline{\lambda}}=1$.
\end{enumerate}
\end{enumerate}
\end{corollary}
\begin{proof}
If $\max_{i\ne j} \left\{|r_i|+|r_j|\right\}\geq 1$, then $s_0>1$. Moreover, excluding a very small set \cite[Theorem 2.1.(a)]{SimonSolomyak_2002Fractals} states that $\Lambda^{\underline{\lambda}}$ contains an interval.

If $\max_{i\ne j} \left\{|r_i|+|r_j|\right\}< 1$  and $s_0>1$, then \cite[Theorem 2.1.(c)]{SimonSolomyak_2002Fractals} states that excluding a very small set $\mathcal{L}(\Lambda^{\underline{\lambda}})>0$.

If $\max_{i\ne j} \left\{|r_i|+|r_j|\right\}< 1$ and $s_0\leq 1$, then Lemma~\ref{lemma:01} implies that $\{\mathcal{S}^{\underline{\lambda} }\}_{\underline{\lambda} \in B}$ is a transversal family and point $(2)$ follows from Theorem~\ref{thm:99} and \cite{KaenmakiEtal_SelfConf_20BLMS}.
\end{proof}

%%%%%%%%%%%%%%%%%%%%%%%%%%%%%%%%%%%%%%%%%%%%%%%%%%%%%%%%%%%%%%%%%%%%%%%%%%%%%%
\section{Preliminaries}\label{sec:97}

Here we summarize the symbolic notation we use, moreover, we recall the bounded distortion property and equivalent characterizations of the Open Set Condition.

%%%%%%%%%%%%%%%%%%%%%%%%%%%%%%%%%%%%%%%%%%%%%%%%%%%%%%%%%%%%%%%%%%%%%%%%%%%%%%
\subsection{Symbolic notation}\label{sec:99}
An element of the symbolic space $\Sigma=\{1,\ldots,m\}^\N$ is an infinite sequence denoted by $\ii=i_1i_2\ldots$. The left shift operator $\sigma$ on $\Sigma$ maps $\ii=i_1i_2\ldots$ to $\sigma\ii=i_2i_3\ldots$. The set of all finite length words is denoted by $\Sigma^\ast:=\bigcup_{n=1}^{\infty }\left\{1, \dots ,m\right\}^{n}$. Elements of $\Sigma^\ast$ are either written as truncations of infinite length words $\ii|n=i_1\ldots i_n$ or as $\iiv=i_1\ldots i_n$. For compositions of maps we use the standard notation $S_{\iiv}=S_{i_1}\circ\ldots\circ S_{i_n}$. For a finite word $\iiv=i_1\ldots i_n\in\Sigma^*$, denote $\iiv^{\infty}$ the infinite sequence for which $\sigma^n\iiv^{\infty}=\iiv^{\infty}$.

Given an IFS $\mathcal{S}=\left\{S_i\right\}_{i=1}^{m}$, the natural projection $\Pi:\Sigma\to\Lambda$ is defined by
\begin{equation*}
\Pi(\ii) :=\lim\limits_{n\to\infty}
S_{\mathbf{i}|n}(x),
\end{equation*}
where $x\in X$ is arbitrary. The Strong Separation Property holds if and only if $\Pi$ is a one-to-one coding. The function $\ii\mapsto\Pi(\ii)$ from $\Sigma$ to $X$ is continuous. When necessary, to show dependence of $\Pi$ on the IFS $\mathcal{S}$, we write $\Pi_{\mathcal{S}}$ or $\Pi_{\underline{\lambda} }:\Sigma\to \Lambda^{\underline{\lambda} }$ when we work with a parameterized family of IFSs $\{\mathcal{S}^{\underline{\lambda} }\}_{\underline{\lambda} \in B}$.

A self-conformal IFS $\mathcal{S}$ satisfies the \texttt{Bounded Distortion Property}: there exists a uniform constant $C_0>0$ such that
\begin{equation}\label{eq:96}
C_0^{-1}<\frac{|S'_{\iiv}(x)|}{|S'_{\iiv}(y)|}<C_0 \;\text{ for every } \iiv\in\Sigma^{\ast} \text{ and } x,y\in X.
\end{equation}

%%%%%%%%%%%%%%%%%%%%%%%%%%%%%%%%%%%%%%%%%%%%%%%%%%%%%%%%%%%%%%%%%%%%%%%%%%%%%%
\subsection{Equivalent characterizations of the Open Set Condition}\label{sec:98}

Recall the separation conditions we introduced in Section~\ref{sec:00}. In the proofs, we will use one more condition.

An IFS $\mathcal{S}$ satisfies the \texttt{Open Set Condition (OSC)} if there is a non-empty open set $U$ such that $S_i(U)\subseteq U$ and $S_i(U)\cap S_j(U)=\emptyset$ for distinct $i,j\in\{1,\ldots,m\}$.

In order to give equivalent characterizations of the OSC, we define the conformal dimension of a self-conformal IFS $\mathcal{S}$ and its attractor $\Lambda$. For $s>0$, the pressure function $P(s)$ is defined by
\begin{equation*}
P(s) = \lim_{n\to\infty} \frac{1}{n} \log \sum_{\iiv\in\Sigma^n} \|S_{\iiv}'\|^s,
\end{equation*}
where $\|\cdot\|$ denotes the supremum norm of a function.

The \texttt{conformal dimension} of $\Lambda$ is the unique solution $s_0$ of the Bowen equation $P(s_0)=0$. The bounded distortion property~\eqref{eq:96} and the mean value theorem imply that there exists a uniform constant $C_1>0$ such that for all $\iiv\in\Sigma^{\ast}$ we have
\begin{equation*}\label{cv88}
C_1^{-1} < \frac{\|S_{\iiv}'\|}{\mathrm{diam}(S_{\iiv}(\Lambda)) } < C_1.
\end{equation*}
Hence, $s_0$ is always an upper bound for $\dim_{\mathrm H}\Lambda$. Moreover, $\mathcal{H}^{s_0}(\Lambda)<\infty$. If $\mathcal{S}$ is a self-similar IFS, then the Bowen equation simply becomes
\begin{equation*}
|r_{1}|^{s_0}+\cdots+|r_{m}|^{s_0}=1.
\end{equation*}
In this case, $s_0$ is called the similarity dimension.

Peres, Rams, Simon and Solomyak \cite[Theorem 1.1]{peres2001equivalence} showed the following equivalences for self-conformal IFSs:
\begin{align}
\mbox{OSC} &\;\Longleftrightarrow\; \mathcal{H}^{s_0}(\Lambda)>0 \;(\text{in particular, } \dim_{\mathrm H}\Lambda=s_0)\nonumber \\
&\;\Longleftrightarrow\;
\inf\left\{
\frac{\sup_{x\in\Lambda}  \left|
	S_{\iiv}(x)-S_{\jjv}(x)
	\right|}{\min\left\{
	\|S'_{\iiv}(\Lambda)\|,
	\|S'_{\jjv}(\Lambda)\|,
	\right\}}
:
\iiv,\jjv\in\Sigma^{\ast},\;
\iiv\neq\jjv
\right\}>0. \label{eq:32}
\end{align}
These equivalences were proved previously in the self-similar case by Bandt and Graf~\cite{BandtGraf_1992self} and Schief~\cite{Shief_SelfSim_94ProcAMS}. It is not difficult to see that~\eqref{eq:32} is also equivalent to
\begin{equation}\label{eq:39}
\inf\Big\{
\|S'_{\iiv}(\Lambda)\|^{-1}
\sup\limits_{x\in \Lambda}
|
S_{\iiv}(x)-S_{\jjv}(x)
|
:
\iiv,\jjv\in\Sigma^{\ast}, \iiv\neq\jjv
\Big\}>0.
\end{equation}

The relationship between the OSC and the WSP is given by the following observation.
\begin{fact}\label{eq:31}
A self-conformal IFS $\mathcal{S}$ on the line (or more generally in $\mathbb{R}^d$) satisfies the $\mathrm{OSC}$ if and only if $\mathcal{S}$ satisfies the $\mathrm{WSP}$ and has {\bf no} \emph{exact overlaps}.
\end{fact}
\begin{proof}
For self-conformal IFSs, it was proved in \cite[Theorem 3.2]{KaenmakiEtal_SelfConf_20BLMS} that
the WSP is equivalent to the \texttt{identity limit criterion}, which holds if
\begin{equation}\label{eq:30}
\inf\Big\{
\|S'_{\iiv}\|^{-1}
\sup\limits_{x\in\Lambda}
\left|
S_{\iiv}(x)-S_{\jjv}(x)
\right|
:
\iiv,\jjv\in\Sigma^{\ast},\,
S_{\iiv}|_{\Lambda}\ne
S_{\jjv}|_{\Lambda}
\Big\}>0.
\end{equation}
Comparing this with \eqref{eq:39} proves the assertion.
\end{proof}

%%%%%%%%%%%%%%%%%%%%%%%%%%%%%%%%%%%%%%%%%%%%%%%%%%%%%%%%%%%%%%%%%%%%%%%%%%%
%%%%%%%%%%%%%%%%%%%%%%%%%%%%%%%%%%%%%%%%%%%%%%%%%%%%%%%%%%%%%%%%%%%%%%%%%%%
\section{Generic conformal IFSs, proof of Theorem \ref{thm:98}}\label{sec:04}

Similarly to parameterized families, let us introduce
\begin{align*}
\mathscr{SSP}_{\beta,\rho}^{m}(X) &:=
\left\{
\mathcal{S}\in\Theta_{\beta,\rho}^{m}(X):
\mbox{SSP  holds for } \mathcal{S}
\right\},  \\
\mathscr{WSP}_{\beta,\rho}^{m}(X) &:=
\left\{
\mathcal{S}\in\Theta_{\beta,\rho}^{m}(X):
\mbox{WSP holds for } \mathcal{S}
\right\}, \\
\mathscr{EO}_{\beta,\rho}^{m}(X) &:=
\left\{
\mathcal{S}\in\Theta_{\beta,\rho}^{m}(X):
\mathcal{S} \mbox{ has an exact overlap}
\right\}.
\end{align*}
Slightly abusing notation, since $m$, $0<\beta<\rho<1$ and the compact interval $X$ are fixed, we suppress them from the notation and write simply $\mathscr{SSP},\, \mathscr{WSP},\, \mathscr{EO}$, moreover, $\mathscr{NOWSP}:=\Theta_{\beta,\rho}^{m}(X)\setminus \mathscr{WSP}$ and $\mathscr{NOSSP}:=\Theta_{\beta,\rho}^{m}(X)\setminus \mathscr{SSP}$.

To prove Theorem~\ref{thm:98} it is enough to show that $\mathscr{SSP}\cup
\mathscr{NOWSP}$ is a dense $G_\delta$ set in $\Theta_{\beta,\rho}^{m}(X)$ and that $\mathscr{EO}$ is a set of first category.

\begin{theorem}\label{thm:extended}
		In the topology on $\mathscr{NOSSP}$ defined by \eqref{eq:38}, the set $\mathscr{NOWSP}$ is dense $G_\delta$.
\end{theorem}

First, let us show that

\begin{lemma}\label{lem:50}
In the topology defined by the distance in \eqref{eq:38},
\begin{equation*}
\mathscr{NOWSP} \mbox{ is a } G_\delta \mbox{ set in }
\mathscr{NOSSP}
\end{equation*}
\end{lemma}
\begin{proof}
For an $\varepsilon>0$, we write
$$
\mathcal{V}_\varepsilon\!:=\!
\Big\{
\mathcal{S}\!\in\!  \Theta_{\beta,\rho}^{m}(X)\!:\!
\exists \iiv,\jjv\in\Sigma^{\ast},\,
S_{\iiv}|_{\Lambda}
\neq
S_{\jjv}|_{\Lambda},\,
\|S'_{\iiv}\|^{-1}\sup_{x\in\Lambda}  \left|
	S_{\iiv}(x)-S_{\jjv}(x)
	\right|\!
<\!
\varepsilon
\Big\}.
$$
By its definition, the set $\mathcal{V}_\varepsilon$ is open in $\Theta_{\beta,\rho}^{m}(X)$.
From \eqref{eq:30}, we get that
\begin{equation*}
  \mathscr{NOWSP}=
  \bigcap\limits_{n=1}^{\infty }
  \mathcal{V}_{1/n},
\end{equation*}
i.e. $\mathscr{NOWSP}$ is a $G_\delta$ set.
\end{proof}
\begin{lemma}\label{lem:45}
 $\mathscr{NOWSP}$ is a dense subset of $\mathscr{NOSSP}$.
\end{lemma}

Theorem~\ref{thm:extended} clearly follows by the combination of Lemma~\ref{lem:50} and Lemma~\ref{lem:45}.

\begin{proof}[Proof of Lemma~\ref{lem:45}]
It is enough to show that for any $\mathcal{G}\notin\mathscr{SSP}$ there exists an $\widetilde{\mathcal{S}}\in\mathscr{NOWSP}$ arbitrarily close to $\mathcal{G}$ in the topology defined by the distance   in \eqref{eq:38}. We achieve this through a succession of steps.

\begin{fact}\label{fact:96}
	Assume $\mathcal{G}\notin\mathscr{SSP}$, i.e. there exist $\mathbf{i},\mathbf{j}\in\Sigma$ with $i_1\neq j_1$ such that
	\begin{equation}\label{eeq:81}
	\Pi_{\mathcal{G}}(\mathbf{i})=\Pi_{\mathcal{G}}(\mathbf{j}).
	\end{equation}
	Then for every $\widetilde{\varepsilon}>0$ there exists an IFS $\widetilde{\mathcal{G}} \in \Theta_{\beta,\rho}^{m}(X)$ satisfying
	\begin{equation}\label{eeq:80}
	\Pi_{\widetilde{\mathcal{G}}}(\mathbf{j})-\Pi_{\widetilde{\mathcal{G}}}(\mathbf{i})>0 \;\text{ and }\; \mathrm{dist}(\mathcal{G},\widetilde{\mathcal{G}})<\widetilde{\varepsilon}.
	\end{equation}
\end{fact}

In the next step, for any $\alpha\in [0,1]$, we take the convex linear combination of $\mathcal{G}$ and $\widetilde{\mathcal{G}}$ from Fact~\ref{fact:96} defined by
\begin{equation*}
\mathcal{S}^\alpha:=
\alpha\mathcal{G}+(1-\alpha)\widetilde{\mathcal{G}}=\big\{S_i^\alpha=\alpha G_i+(1-\alpha)\widetilde{G}_i \big\}_{i=1}^{m}.
\end{equation*}

\begin{fact}\label{fact:97}
Assume $\mathcal{G},\,\widetilde{\mathcal{G}}\in \Theta_{\beta,\rho}^{m}(X)$ satisfy \eqref{eeq:81} and \eqref{eeq:80}, respectively, from Fact~\ref{fact:96}. Then $\mathcal{S}^\alpha\in \Theta_{\beta,\rho}^{m}(X)$ for all $\alpha\in[0,1]$ and there exists an $\alpha^\ast\in[0,1]$ such that
\begin{equation*}
\mathcal{S}^{\alpha^\ast}\in\mathscr{U},
\end{equation*}
where
\begin{multline}\label{eeq:82}
\mathscr{U}:=
\big\{
\mathcal{S}\in
\Theta_{\beta,\rho}^{m}(X):
\exists\, \widetilde{x}\in X,
\exists\,  \pmb{\omega},\pmb{\tau}\in\Sigma^*, \\
\omega_1\neq\tau_1,\; \omega_{|\pmb{\omega}|}\neq
\tau_{|\pmb{\tau}|},\;
\widetilde{x}=S_{\pmb{\omega}}(\widetilde{x})
=S_{\pmb{\tau}}(\widetilde{x})
\big\}.
\end{multline}
\end{fact}

\begin{fact}\label{fact:98}
	For every $\widetilde{\varepsilon}>0$ and every $\mathcal{S}\in\mathscr{U}$ there exists an IFS $\widetilde{\mathcal{S}}\in \Theta_{\beta,\rho}^{m}(X)$ such that
	\begin{equation}\label{cv83}
	\mathrm{dist}\big(\mathcal{S},
	\widetilde{\mathcal{S}}\big)<\widetilde{\varepsilon} \;\mbox{ and }\;
	\widetilde{\mathcal{S}}\in\mathscr{R},
	\end{equation}
	where
\begin{multline*}
\mathscr{R}:=
\Big\{
\mathcal{S}\in \Theta_{\beta,\rho}^{m}(X)
:
\exists\, \widetilde{x}\in X,
\exists\,  \pmb{\omega},\pmb{\tau}\in\Sigma^*, \\
\omega_1\neq\tau_1,\; \omega_{|\pmb{\omega}|}\neq
\tau_{|\pmb{\tau}|},\;
\widetilde{x}=S_{\pmb{\omega}}(\widetilde{x})
=S_{\pmb{\tau}}(\widetilde{x}),\;
\frac{\log |S'_{\pmb{\omega}}(\widetilde{x})|}{\log |S'_{\pmb{\tau}}(\widetilde{x})|}
\not\in\mathbb{Q} \Big\}.
\end{multline*}
In particular, this holds for the IFS $\mathcal{S}^{\alpha^\ast}$ constructed in Fact~\ref{fact:97}.	
\end{fact}

To conclude the proof of Lemma~\ref{lem:45}, we show the following.

\begin{fact}\label{fact:99}
	If $\mathcal{S}\in\mathscr{R}$, then WSP does not hold for $\mathcal{S}$, i.e. $\mathscr{R} \subset \mathscr{NOWSP}$.
\end{fact}

We prove all the facts separately in Subsection~\ref{sec:31}.
\end{proof}

\begin{lemma}\label{lemma:40}
	The set $\mathscr{EO}_{\beta,\rho}^{m}(X)$ is a set of first category, i.e. it is the union of countably many nowhere dense subsets of $\Theta_{\beta,\rho}^{m}(X)$.
\end{lemma}
\begin{proof}
We can write $\mathscr{EO}$ as the countable union
\begin{equation*}
\mathscr{EO} = \bigcup_{\stackrel{\iiv,\jjv\in\Sigma^\ast}{i_1\neq j_1}} B_{\iiv,\jjv},
\end{equation*}
where $B_{\iiv,\jjv}= \big\{\mathcal{S}\in\Theta_{\beta,\rho}^{m}(X) : S_{\iiv}\equiv S_{\jjv},\, i_1\neq j_1 \big\}$. Hence, it is enough to show that each $B_{\iiv,\jjv}$ is nowhere dense.

Since $S_{\iiv}\equiv S_{\jjv}$, we have that $\Pi_{\mathcal{S}}(\iiv^{\infty})=\Pi_{\mathcal{S}}(\jjv^\infty)$. Then Fact~\ref{fact:96} implies that for every $\varepsilon>0$ there exists an IFS $\widetilde{\mathcal{S}} \in \Theta_{\beta,\rho}^{m}(X)$ satisfying
\begin{equation*}
\Pi_{\widetilde{\mathcal{S}}}(\jjv^\infty)-\Pi_{\widetilde{\mathcal{S}}}(\iiv^\infty)>0 \;\text{ and }\; \mathrm{dist}(\mathcal{S},\widetilde{\mathcal{S}})<\varepsilon. \text{ In particular, } \widetilde{\mathcal{S}}\notin B_{\iiv,\jjv}.
\end{equation*}
Thus, the complement of $B_{\iiv,\jjv}$ is necessarily open and dense, in other words $B_{\iiv,\jjv}$ is a closed nowhere dense set.
\end{proof}
\begin{proof}[Proof of Theorem~\ref{thm:98}]
It is clear that $\Theta_{\beta,\rho}^{m}(X)=\mathscr{SSP}\bigcup\mathscr{NOSSP}$. It is easy to see that $\mathscr{SSP}$ is an open set in $\Theta_{\beta,\rho}^{m}(X)$. Thus,
the theorem follows directly from Theorem~\ref{thm:extended} and Lemma~\ref{lemma:40}.
\end{proof}

%%%%%%%%%%%%%%%%%%%%%%%%%%%%%%%%%%%%%%%%%%%%%%%%%%%%%%%%%%%%%%%%%%%%%%%%%%%
%%%%%%%%%%%%%%%%%%%%%%%%%%%%%%%%%%%%%%%%%%%%%%%%%%%%%%%%%%%%%%%%%%%%%%%%%%%
\subsection{Proof of facts in the proof of Lemma~\ref{lem:45}}\label{sec:31}
The following claim is used in the proof of Fact \ref{fact:99}.

\begin{claim}\label{cv93}
Let $\mathcal{S}=(S_1, \dots ,S_m)\in \Theta_{\beta,\rho}^{m}(X)$.
Using the notation of Definition \ref{def:95}
  and  the bounded distortion constants $C_0$ introduced in \eqref{eq:96} we have
  \begin{equation*}\label{cv91}
  \frac{\mathrm{diam}( S_{\iiv^-}(X))}{\mathrm{diam}( S_{\iiv}(X))}
   >
   1+C_{0}^{-1} \cdot \frac{1-\rho}{2\rho}=:\tau,\quad
   \forall \iiv\in\Sigma^*.
  \end{equation*}
\end{claim}
\begin{proof}
  Let $k:=\iiv_{|\iiv|}$. That is $S_{\iiv}=S_{\iiv^-}\circ S_k$. We can choose an interval $I \subset X$ such that
  \begin{equation*}\label{cv90}
    \mathrm{int}(S_k(X))\cap \mathrm{int}(I)=\emptyset
    \mbox{ and }
    \mathrm{diam}(I)>\mathrm{diam}(X) \cdot \frac{1-\rho}{2}.
  \end{equation*}
  Hence,
 \begin{align*}
	% to remove numbering (before each equation)
	\frac{\mathrm{diam}( S_{\iiv^-}(X))}{\mathrm{diam}( S_{\iiv}(X))}    & \geq
	\frac{\mathrm{diam}( S_{\iiv^-}(S_k(X))+ \mathrm{diam}(S_{\iiv^-}(I)))}
	{\mathrm{diam}( S_{\iiv^-}(S_k(X)))}
	\\
	&= 1+
	\frac{ \mathrm{diam}(S_{\iiv^-}(I)))}
	{\mathrm{diam}( S_{\iiv^-}(S_k(X)))}
	\geq
	1+\frac{ \inf\limits_{x\in X}|S'_{\iiv^-}(x)|\cdot \mathrm{diam}(I)}
	{\|S'_{\iiv^-}\| \cdot
		\mathrm{diam}(S_k(X) ) } \\
	& >  1+
	C_{0}^{-1} \cdot \frac{\mathrm{diam}(X) \cdot \frac{1-\rho}{2}}
	{\rho \cdot \mathrm{diam}(X) }.
\end{align*}
\end{proof}

\begin{proof}[Proof of Fact~\ref{fact:96}]
Let $\mathcal{G}=\{G_1,\ldots,G_m\}\in \Theta_{\beta,\rho}^{m}(X)$ be such that there exist $\mathbf{i},\mathbf{j}\in\Sigma$ with $i_1\neq j_1$ for which
\begin{equation}\label{eeq:65}
\Pi_{\mathcal{G}}(\mathbf{i})=\Pi_{\mathcal{G}}(\mathbf{j}).
\end{equation}
The infinite words $\mathbf{i}$ and $\mathbf{j}$ are fixed from now. Without loss of generality we may assume that $0<\beta < |G'_i(x)| < \rho<1$. Otherwise, if $|G'_i(x)|$ attains $\beta$ or $\rho$ for some $x\in X$ and $i\in\{1,\ldots, m\}$, then consider the IFS $\mathcal{G}_\varepsilon=\{(1-\varepsilon)G_i(x)\pm\varepsilon\frac{\beta+\rho}{2}x\}$, where $\pm$ agrees with the sign of $G_i'(x)$ and by choosing $\varepsilon$ such that $\mathrm{dist}(\mathcal{G}_\varepsilon, \mathcal{G})<\widetilde{\varepsilon}$. Using the function
\begin{equation}\label{eeq:64}
T_{\delta,y}(x):=
\begin{cases}
\varepsilon\cdot (x-(y-\delta))^4\cdot(x-(y+\delta))^4, &\text{if } x\in[y-\delta,y+\delta] \\
0, &\text{otherwise},
\end{cases}
\end{equation}
we perturb $\mathcal{G}$ to an IFS $\widetilde{\mathcal{G}}=\{\widetilde{G}_1,\ldots,\widetilde{G}_m\}$ defined by
\begin{equation}\label{eeq:63}
\widetilde{G}_k(x):=
\begin{cases}
G_{i_1}(x)+T_{\delta,\Pi_{\mathcal{G}}(\sigma \ii)}(x), &\text{if } k=i_1; \\
G_{k}(x), &\text{otherwise}.
\end{cases}
\end{equation}
Clearly, $T_{\delta,y}\in\mathcal{C}^2$ and simple calculations give
\begin{equation}\label{eeq:62}
\|T_{\delta,y}\|_{\infty} = \varepsilon \delta^8,\ \quad  \|T'_{\delta,y}\|_{\infty}\leq 8\cdot \varepsilon\delta^7 \,\text{ and }\,
\|T''_{\delta,y}\|_{\infty}\leq 54\cdot \varepsilon\delta^6.
\end{equation}
For fixed $\widetilde{\varepsilon}>0$ and $\delta>0$, the parameter $\varepsilon$ in \eqref{eeq:64} is chosen so small that
\begin{equation}\label{eeq:61}
\widetilde{\mathcal{G}}\in \Theta_{\beta,\rho}^{m}(X), \text{ i.e. } \beta\leq \big|\widetilde{G}'_{k}(x)\big| \leq\rho \text{ for every } k=1,\ldots,m \text{ and } x\!\in\! X
\end{equation}
and
\begin{equation*}
\mathrm{dist}(\mathcal{G},\widetilde{\mathcal{G}})<\widetilde{\varepsilon}
\end{equation*}
hold simultaneously. The appropriate $\delta>0$ is specified later in the proof.

For this IFS $\widetilde{\mathcal{G}}$, we claim that
\begin{equation*}
\Pi_{\widetilde{\mathcal{G}}}(\mathbf{i})-\Pi_{\widetilde{\mathcal{G}}}(\mathbf{j})>0.
\end{equation*}
First observe that \eqref{eeq:65} implies that
\begin{equation*}
\Pi_{\widetilde{\mathcal{G}}}(\mathbf{i})-\Pi_{\widetilde{\mathcal{G}}}(\mathbf{j}) = \big(\Pi_{\widetilde{\mathcal{G}}}(\mathbf{i}) - \Pi_{\mathcal{G}}(\mathbf{i})\big) + \big( \Pi_{\mathcal{G}}(\mathbf{j}) -\Pi_{\widetilde{\mathcal{G}}}(\mathbf{j}) \big).
\end{equation*}
Since $j_1\neq i_1$, using the definition \eqref{eeq:63} of $\widetilde{\mathcal{G}}$, we bound
\begin{align*}
\big|\Pi_{\mathcal{G}}(\mathbf{j}) -\Pi_{\widetilde{\mathcal{G}}}(\mathbf{j})\big| &= \big| G_{j_1}\big( \Pi_{\mathcal{G}}(\sigma\mathbf{j}) \big) - \widetilde{G}_{j_1}\big( \Pi_{\widetilde{\mathcal{G}}}(\sigma\mathbf{j}) \big)\big| \\
& = \big| G_{j_1}\big( \Pi_{\mathcal{G}}(\sigma\mathbf{j}) \big) - G_{j_1}\big( \Pi_{\widetilde{\mathcal{G}}}(\sigma\mathbf{j}) \big)\big| \leq \rho\cdot \big|\Pi_{\mathcal{G}}(\sigma\mathbf{j}) -\Pi_{\widetilde{\mathcal{G}}}(\sigma\mathbf{j})\big|.
\end{align*}
Similarly,
\begin{align*}
\Pi_{\mathcal{G}}(\mathbf{i}) -\Pi_{\widetilde{\mathcal{G}}}(\mathbf{i}) &=  T_{\delta,\Pi_{\mathcal{G}}(\sigma \ii)}( \Pi_{\mathcal{G}}\big(\sigma \ii) \big) + \widetilde{G}_{i_1}\big( \Pi_{\mathcal{G}}(\sigma\mathbf{i}) \big) - \widetilde{G}_{i_1}\big( \Pi_{\widetilde{\mathcal{G}}}(\sigma\mathbf{i}) \big) \\
&\geq \varepsilon \delta^8 - \rho \big|\Pi_{\mathcal{G}}(\sigma\mathbf{i}) -\Pi_{\widetilde{\mathcal{G}}}(\sigma\mathbf{i})\big|,
\end{align*}
where we used \eqref{eeq:62} and \eqref{eeq:61}. Hence,
\begin{equation}\label{eeq:60}
\Pi_{\widetilde{\mathcal{G}}}(\mathbf{i})-\Pi_{\widetilde{\mathcal{G}}}(\mathbf{j}) \geq \varepsilon \delta^8 -\rho\cdot \left( \big|\Pi_{\mathcal{G}}(\sigma\mathbf{i}) -\Pi_{\widetilde{\mathcal{G}}}(\sigma\mathbf{i})\big| + \big|\Pi_{\mathcal{G}}(\sigma\mathbf{j}) -\Pi_{\widetilde{\mathcal{G}}}(\sigma\mathbf{j})\big| \right).
\end{equation}
Also observe that by a simple induction argument, for every $\pmb{\omega}\in\Sigma$
\begin{equation}\label{eeq:59}
\big|\Pi_{\mathcal{G}}(\pmb{\omega}) -\Pi_{\widetilde{\mathcal{G}}}(\pmb{\omega})\big| \leq \varepsilon\cdot \delta^8+\rho \cdot \big|\Pi_{\mathcal{G}}(\sigma\pmb{\omega}) -\Pi_{\widetilde{\mathcal{G}}}(\sigma\pmb{\omega})\big| \leq \frac{\varepsilon\cdot \delta^8}{1-\rho}\, .
\end{equation}

To continue, we distinguish between four different cases. \\
\textbf{Case I.} Assume that for every $\ell\geq 2$
\begin{equation*}
\Pi_{\mathcal{G}}(\sigma^\ell\ii) \neq \Pi_{\mathcal{G}}(\sigma\ii) \;\text{ and }\; \Pi_{\mathcal{G}}(\sigma^\ell\jj) \neq \Pi_{\mathcal{G}}(\sigma\ii).
\end{equation*}
Choose $L\geq 2$ to be the smallest integer such that
\begin{equation*}
1-\rho-2\rho^L>0
\end{equation*}
and let
\begin{equation*}
\delta:= \frac{1}{2}\, \min\left\{ \big|\Pi_{\mathcal{G}}(\sigma^\ell\ii) - \Pi_{\mathcal{G}}(\sigma\ii)\big|,\, \big|\Pi_{\mathcal{G}}(\sigma^\ell\jj) - \Pi_{\mathcal{G}}(\sigma\ii)\big|:\, \ell=2,\ldots,L\right\}>0.
\end{equation*}
This choice of $\delta$ implies that $\widetilde{G}_{i_\ell}(\Pi_{\mathcal{G}}(\sigma^{\ell}\ii))= G_{i_\ell}(\Pi_{\mathcal{G}}(\sigma^{\ell}\ii))$ and $\widetilde{G}_{j_\ell}(\Pi_{\mathcal{G}}(\sigma^{\ell}\jj))= G_{j_\ell}(\Pi_{\mathcal{G}}(\sigma^{\ell}\jj))$ for every $\ell=2,\ldots,L$. Therefore, induction immediately gives
\begin{align*}
\big|\Pi_{\mathcal{G}}(\sigma\ii) -\Pi_{\widetilde{\mathcal{G}}}(\sigma\ii)\big| &\leq \rho\cdot \big|\Pi_{\mathcal{G}}(\sigma^2\ii) -\Pi_{\widetilde{\mathcal{G}}}(\sigma^2\ii)\big| \leq \ldots \\
&\leq \rho^{L-1}\cdot \big|\Pi_{\mathcal{G}}(\sigma^L\ii) -\Pi_{\widetilde{\mathcal{G}}}(\sigma^L\ii)\big| \leq \rho^{L-1}\frac{\varepsilon\cdot \delta^8}{1-\rho}\, ,
\end{align*}
where we used \eqref{eeq:59} in the last step. The same is true for $\big|\Pi_{\mathcal{G}}(\sigma\jj) -\Pi_{\widetilde{\mathcal{G}}}(\sigma\jj)\big|$. Plugging this back into \eqref{eeq:60} yields
\begin{equation*}
\Pi_{\widetilde{\mathcal{G}}}(\mathbf{i})-\Pi_{\widetilde{\mathcal{G}}}
(\mathbf{j}) \geq \varepsilon \delta^8 -2 \rho^L \frac{\varepsilon\cdot \delta^8}{1-\rho} = (1-\rho-2\rho^L) \frac{\varepsilon\cdot \delta^8}{1-\rho} >0
\end{equation*}
by our choice of $L$. This concludes the proof in Case I. \\
\textbf{Case II.} Assume that $L\geq 2$ is the smallest integer such that
\begin{equation*}
\Pi_{\mathcal{G}}(\sigma^L\ii) = \Pi_{\mathcal{G}}(\sigma\ii),
\end{equation*}
i.e. $\Pi_{\mathcal{G}}(\sigma\ii)$ is a fixed point of $G_{i_2}\circ\dots\circ G_{i_L}$. Thus, without loss of generality we can redefine $\ii=i_1(i_2\ldots i_L)^\infty$, where $(i_2\ldots i_L)^\infty\in\Sigma$ is obtained by concatenating $(i_2\ldots i_L)$ after each other infinitely many times. Also assume that for every $n\geq 2$
\begin{equation*}
\Pi_{\mathcal{G}}(\sigma^n\jj) \neq \Pi_{\mathcal{G}}(\sigma\ii).
\end{equation*}
Choose $N\geq 2$ to be the smallest integer such that
\begin{equation}\label{eeq:56}
\frac{1}{1+\rho^{L-1}}-\frac{\rho^N}{1-\rho}>0.
\end{equation}
Let
\begin{equation*}
\delta:= \frac{1}{2}\; \min_{\substack{2\leq\ell\leq L-1 \\ 2\leq n\leq N}}\,\left\{ \big|\Pi_{\mathcal{G}}(\sigma^\ell\ii) - \Pi_{\mathcal{G}}(\sigma\ii)\big|,\, \big|\Pi_{\mathcal{G}}(\sigma^n\jj) - \Pi_{\mathcal{G}}(\sigma\ii)\big|\right\}>0.
\end{equation*}

This choice of $\delta$ implies that $\widetilde{G}_{j_n}(\Pi_{\mathcal{G}}(\sigma^{n}\jj))= G_{j_n}(\Pi_{\mathcal{G}}(\sigma^{n}\jj))$ for every $n=2,\ldots,N$. Therefore, using the same argument as in Case I we bound
\begin{equation*}
\big|\Pi_{\mathcal{G}}(\jj) -\Pi_{\widetilde{\mathcal{G}}}(\jj)\big| \leq \rho^N \frac{\varepsilon\cdot\delta^8}{1-\rho}.
\end{equation*}

To bound $\big|\Pi_{\mathcal{G}}(\ii) -\Pi_{\widetilde{\mathcal{G}}}(\ii)\big|$ from below, first assume that $i_L\neq i_1$. Then the choice of $\delta$ and $i_L\neq i_1$ implies that $\widetilde{G}_{i_\ell}(\Pi_{\mathcal{G}}(\sigma^{\ell}\ii))= G_{i_\ell}(\Pi_{\mathcal{G}}(\sigma^{\ell}\ii))$ for $\ell=2,\ldots,L$. Thus
\begin{equation*}
\Pi_{\mathcal{G}}(\sigma\ii) -\Pi_{\widetilde{\mathcal{G}}}(\sigma\ii) =\widetilde{G}_{i_2}\circ\dots\circ \widetilde{G}_{i_L}\big(\Pi_{\mathcal{G}}(\sigma\ii)\big) -
\widetilde{G}_{i_2}\circ\dots\circ \widetilde{G}_{i_L}\big(\Pi_{\widetilde{\mathcal{G}}}(\sigma\ii)\big),
\end{equation*}
where we used that $\sigma\ii=(i_2\ldots i_L)^\infty$. Since $\widetilde{G}_{i_2}\circ\dots\circ \widetilde{G}_{i_L}$ is a monotone strict contraction, we obtain that $\Pi_{\mathcal{G}}(\sigma\ii)=\Pi_{\widetilde{\mathcal{G}}}(\sigma\ii)$. As a result
\begin{equation*}
\Pi_{\widetilde{\mathcal{G}}}(\mathbf{i}) -\Pi_{\mathcal{G}}(\mathbf{i}) =  T_{\delta,\Pi_{\mathcal{G}}(\sigma \ii)}( \Pi_{\mathcal{G}}\big(\sigma \ii) \big) + G_{i_1}\big( \Pi_{\mathcal{G}}(\sigma\mathbf{i}) \big) - G_{i_1}\big( \Pi_{\widetilde{\mathcal{G}}}(\sigma\mathbf{i}) \big) = \varepsilon\delta^8.
\end{equation*}
Now assume that $i_L = i_1$, i.e. $\ii=(i_1\ldots i_{L-1})^\infty$. Then
\begin{align*}
\Pi_{\widetilde{\mathcal{G}}}(\ii) &-\Pi_{\mathcal{G}}(\ii) \\
&=  \varepsilon\delta^8 + G_{i_1}\!\circ\dots\circ G_{i_{L-1}}\big(\Pi_{\mathcal{G}}(\sigma^{L-1}\ii)\big) -
G_{i_1}\!\circ\dots\circ G_{i_{L-1}}\big(\Pi_{\widetilde{\mathcal{G}}}(\sigma^{L-1}\ii)\big)  \\
&\geq \varepsilon\delta^8 - \rho^{L-1} \big|\Pi_{\widetilde{\mathcal{G}}}(\ii) -\Pi_{\mathcal{G}}(\ii)\big|.
\end{align*}
In summary, we showed that
\begin{equation*}
\Pi_{\widetilde{\mathcal{G}}}(\ii) -\Pi_{\widetilde{\mathcal{G}}}(\jj) \geq \left(\frac{1}{1+\rho^{L-1}} - \frac{\rho^{N}}{1-\rho}\right)\cdot \varepsilon\delta^8 >0
\end{equation*}
by \eqref{eeq:56}. This concludes the proof in Case II. \\
\textbf{Case III.} Assume that $L\geq 2$ is the smallest integer such that
\begin{equation*}
\Pi_{\mathcal{G}}(\sigma^L\ii) = \Pi_{\mathcal{G}}(\sigma\ii),
\end{equation*}
and there also exists an $M$ such that
\begin{equation*}
\Pi_{\mathcal{G}}(\sigma^M\jj) = \Pi_{\mathcal{G}}(\sigma\ii).
\end{equation*}
We may assume in this case without loss of generality that  for any $N\geq 2$ satisfying $\Pi_{\mathcal{G}}(\sigma^N\jj) = \Pi_{\mathcal{G}}(\sigma\ii)$, we have that $j_N\neq i_1$.

Indeed, if $M$ can be chosen such that $j_M=i_1$, then $\Pi_{\mathcal{G}}(\sigma^{M-1}\jj)=\Pi_{\mathcal{G}}(\ii)$ which equals $\Pi_{\mathcal{G}}(\jj)$ by \eqref{eeq:65}. Hence, $\Pi_{\mathcal{G}}(\jj)$ is a fixed point of $G_{j_1}\circ\dots\circ G_{j_{M-1}}$. If $M$ is the smallest such index then without loss of generality we may redefine $\jj=(j_1\ldots j_{M-1})^\infty$ and in this case, for any $N\geq 2$ satisfying $\Pi_{\mathcal{G}}(\sigma^N\jj) = \Pi_{\mathcal{G}}(\sigma\ii)$, we have that $j_N\neq i_1$.

As in Case II, choose $N\geq 2$ to be the smallest integer such that
\begin{equation*}
\frac{1}{1+\rho^{L-1}}-\frac{\rho^N}{1-\rho}>0.
\end{equation*}
Let
\begin{align*}
\delta:= \frac{1}{2}\min \big\{ &\big|\Pi_{\mathcal{G}}(\sigma^\ell\ii) - \Pi_{\mathcal{G}}(\sigma\ii)\big|,\, \big|\Pi_{\mathcal{G}}(\sigma^n\jj) - \Pi_{\mathcal{G}}(\sigma\ii)\big|: \\
& 2\leq\ell\leq L-1,\, 1\leq n\leq N \,\text{ and }\, \Pi_{\mathcal{G}}(\sigma^n\jj) \neq \Pi_{\mathcal{G}}(\sigma\ii)\big\}>0.
\end{align*}
The argument to bound $\big|\Pi_{\mathcal{G}}(\ii) -\Pi_{\widetilde{\mathcal{G}}}(\ii)\big|$ from below is the same as in Case~II. It remains to bound $\big|\Pi_{\mathcal{G}}(\jj) -\Pi_{\widetilde{\mathcal{G}}}(\jj)\big|$.

The choice of $\delta$ and the fact that $j_n\neq i_1$ whenever $\Pi_{\mathcal{G}}(\sigma^n\jj) = \Pi_{\mathcal{G}}(\sigma\ii)$, implies that $\widetilde{G}_{j_n}(\Pi_{\mathcal{G}}(\sigma^{n}\jj))= G_{j_n}(\Pi_{\mathcal{G}}(\sigma^{n}\jj))$ for $n=1,\ldots,N$. Hence, using \eqref{eeq:59}, we obtain that
\begin{equation*}
\big|\Pi_{\mathcal{G}}(\jj) -\Pi_{\widetilde{\mathcal{G}}}(\jj)\big| =  \big|G_{j_1\ldots j_N}\big(\Pi_{\mathcal{G}}(\sigma^{N}\jj)\big) -
G_{j_1\ldots j_N}\big(\Pi_{\widetilde{\mathcal{G}}}(\sigma^{N}\jj)\big) \big| \leq \rho^N \frac{\varepsilon\cdot\delta^8}{1-\rho}.
\end{equation*}
The conclusion in Case III is the same as in Case II:
\begin{equation*}
\Pi_{\widetilde{\mathcal{G}}}(\ii) -\Pi_{\widetilde{\mathcal{G}}}(\jj) \geq \left(\frac{1}{1+\rho^{L-1}} - \frac{\rho^{N}}{1-\rho}\right)\cdot \varepsilon\delta^8 >0.
\end{equation*}
\textbf{Case IV.} Finally, assume that for every $\ell\geq 2$
\begin{equation*}
\Pi_{\mathcal{G}}(\sigma^\ell\ii) \neq \Pi_{\mathcal{G}}(\sigma\ii)
\end{equation*}
and there exists an $M$ such that
\begin{equation*}
\Pi_{\mathcal{G}}(\sigma^M\jj) = \Pi_{\mathcal{G}}(\sigma\ii).
\end{equation*}
Again, similarly to Case III, we may assume without loss of generality that for every $N\geq2$ with $\Pi_{\mathcal{G}}(\sigma^N\jj) = \Pi_{\mathcal{G}}(\sigma\ii)$, we have that $j_N\neq i_1$.
Choose $L\geq 2$ to be the smallest integer such that
\begin{equation*}
1-\rho-2\rho^L>0.
\end{equation*}
Let
\begin{align*}
\delta:= \frac{1}{2}\min \big\{ &\big|\Pi_{\mathcal{G}}(\sigma^\ell\ii) - \Pi_{\mathcal{G}}(\sigma\ii)\big|,\, \big|\Pi_{\mathcal{G}}(\sigma^n\jj) - \Pi_{\mathcal{G}}(\sigma\ii)\big|: \\
& 2\leq\ell,n\leq L \,\text{ and }\, \Pi_{\mathcal{G}}(\sigma^n\jj) \neq \Pi_{\mathcal{G}}(\sigma\ii)\big\}>0.
\end{align*}
Similarly to Case III, we obtain that
\begin{equation*}
\big|\Pi_{\mathcal{G}}(\jj) -\Pi_{\widetilde{\mathcal{G}}}(\jj)\big| =  \big|G_{j_1\ldots j_N}\big(\Pi_{\mathcal{G}}(\sigma^{N}\jj)\big) -
G_{j_1\ldots j_N}\big(\Pi_{\widetilde{\mathcal{G}}}(\sigma^{N}\jj)\big) \big| \leq \rho^L \frac{\varepsilon\cdot\delta^8}{1-\rho}.
\end{equation*}
On the other hand, similarly to Case I,
\begin{equation*}
\big|\Pi_{\mathcal{G}}(\sigma\ii) -\Pi_{\widetilde{\mathcal{G}}}(\sigma\ii)\big|\leq \rho^{L-1}\frac{\varepsilon\cdot \delta^8}{1-\rho}\, ,
\end{equation*}
and so
\begin{align*}
\Pi_{\widetilde{\mathcal{G}}}(\mathbf{i})-\Pi_{\widetilde{\mathcal{G}}}(\mathbf{j}) & \geq \varepsilon \delta^8 -\rho\cdot\big|\Pi_{\mathcal{G}}(\sigma\mathbf{i}) -\Pi_{\widetilde{\mathcal{G}}}(\sigma\mathbf{i})\big| - \big|\Pi_{\mathcal{G}}(\mathbf{j}) -\Pi_{\widetilde{\mathcal{G}}}(\mathbf{j})\big| \\
& \geq  (1-\rho-2\rho^L) \frac{\varepsilon\cdot \delta^8}{1-\rho} >0 .
\end{align*}

%By redefining $\widetilde{\mathcal{G}}=\{\widetilde{G}_1,\ldots,\widetilde{G}_m\}$ to be perturbed around $\Pi_{\mathcal{G}}(\sigma \jj)$
%\begin{equation*}
%\widetilde{G}_k(x):=
%\begin{cases}
%G_{j_1}(x)-T_{\delta,\Pi_{\mathcal{G}}(\sigma \jj)}(x), &\text{if } k=j_1; \\
%G_{k}(x), &\text{otherwise,}
%\end{cases}
%\end{equation*}
%we can interchange the role of $\ii$ and $\jj$, and argue exactly as in Case II.
\end{proof}

\begin{proof}[Proof of Fact~\ref{fact:97}]
Recall, $\mathcal{S}^\alpha=\alpha\mathcal{G}+(1-\alpha)\widetilde{\mathcal{G}}$, where $\mathcal{G},\widetilde{\mathcal{G}}\in\Theta_{\beta,\rho}^{m}(X)$ are as in Fact~\ref{fact:96}. The definition \eqref{eeq:63} of $\widetilde{G}_i$ implies that $G'_i(x)\cdot(\widetilde{G}_i)'(x)>0$ for all $x\in X$ and $i\in\{1,\ldots,m\}$. Hence, $\mathcal{S}^\alpha \in\Theta_{\beta,\rho}^{m}(X)$ for every $\alpha\in[0,1]$.

%\begin{figure}[h]
%\includegraphics[scale=0.12, angle=90]{CommonFixedPoint}
%\caption{Finding the common fixed point $\widetilde{x}$ at $\alpha^\ast$.}
%\end{figure}

\begin{figure}[h]
\begin{tikzpicture}[scale=0.65,line cap=round,line join=round,x=1cm,y=1cm]
%\draw [color=cqcqcq,, xstep=1cm,ystep=1cm] (-0.1,-1.3) grid (20.1,7.1);
\clip(-0.3,-1.3) rectangle (20.3,7.3);

%pirosak
\draw [line width=2pt,color=ccqqqq] (1,0.6)-- (6.6,0.6);
\draw [color=ccqqqq](1,0.6) node[anchor=east] {$I_{\mathbf i|k}^1$};
\draw [line width=2pt,color=ccqqqq] (1,0.7)-- (1,0.5);
\draw [line width=2pt,color=ccqqqq] (6.6,0.7)-- (6.6,0.5);

\draw [line width=2pt,color=ccqqqq] (5,5.2)-- (6,5.2);
\draw [color=ccqqqq](6,5.2) node[anchor=west] {$I_{\pmb{\omega}}^1$};
\draw [line width=2pt,color=ccqqqq] (5,5.3)-- (5,5.1);
\draw [line width=2pt,color=ccqqqq] (6,5.3)-- (6,5.1);

\draw [line width=2pt,color=ccqqqq] (14.2,2.6)-- (15.2,2.6);
\draw [color=ccqqqq](14.2,2.6) node[anchor=east] {$I_{\pmb{\omega}}^0$};
\draw [line width=2pt,color=ccqqqq] (14.2,2.7)-- (14.2,2.5);
\draw [line width=2pt,color=ccqqqq] (15.2,2.7)-- (15.2,2.5);

\draw [line width=2pt,color=ccqqqq] (11.6,0.6)-- (15.4,0.6);
\draw [color=ccqqqq](11.6,0.6) node[anchor=east] {$I_{\mathbf i|k}^0$};
\draw [line width=2pt,color=ccqqqq] (11.6,0.7)-- (11.6,0.5);
\draw [line width=2pt,color=ccqqqq] (15.4,0.7)-- (15.4,0.5);

\draw [color=ccqqqq](5.8,6) node[anchor=south] {$i(1)$};
\draw [color=ccqqqq](14.8,6) node[anchor=south] {$i(0)$};
\draw [line width=1.3pt,dash pattern=on 2pt off 3pt,color=ccqqqq] (14.8,6)-- (14.8,0);
\draw [line width=1.3pt,dash pattern=on 2pt off 3pt,color=ccqqqq] (5.8,6)-- (5.8,0);

\draw [line width=1.2pt, color=ccqqqq] (5.8,5.2) .. controls (8,3.2) and (12,4.8) .. node[pos=0.3, above] {$\Pi^{\alpha}(\pmb{\omega}^{\infty})$} (14.8,2.6);

%kekek
\draw [color=qqqqff](6.4,1.2) node[anchor= west] {$I_{\mathbf j|k}^1$};
\draw [line width=2pt,color=qqqqff] (1.6,1.2)-- (6.4,1.2);
\draw [line width=2pt,color=qqqqff] (6.4,1.3)-- (6.4,1.1);
\draw [line width=2pt,color=qqqqff] (1.6,1.3)-- (1.6,1.1);

\draw [color=qqqqff](19.2,1.2) node[anchor= west] {$I_{\mathbf j|k}^0$};
\draw [line width=2pt,color=qqqqff] (15.8,1.2)-- (19.2,1.2);
\draw [line width=2pt,color=qqqqff] (15.8,1.3)-- (15.8,1.1);
\draw [line width=2pt,color=qqqqff] (19.2,1.3)-- (19.2,1.1);

\draw [color=qqqqff](2.2,5.2) node[anchor=east] {$I_{\pmb{\tau}}^1$};
\draw [line width=2pt,color=qqqqff] (2.2,5.2)-- (3,5.2);
\draw [line width=2pt,color=qqqqff] (3,5.3)-- (3,5.1);
\draw [line width=2pt,color=qqqqff] (2.2,5.3)-- (2.2,5.1);

\draw [color=qqqqff](17,2.6) node[anchor=west] {$I_{\pmb{\tau}}^0$};
\draw [line width=2pt,color=qqqqff] (16.2,2.6)-- (17,2.6);
\draw [line width=2pt,color=qqqqff] (16.2,2.7)-- (16.2,2.5);
\draw [line width=2pt,color=qqqqff] (17,2.7)-- (17,2.5);

\draw [color=qqqqff](2.4,6) node[anchor=south] {$j(1)$};
\draw [color=qqqqff](16.4,6) node[anchor=south] {$j(0)$};
\draw [line width=1.3pt,dash pattern=on 2pt off 3pt,color=qqqqff] (2.4,6)-- (2.4,0);
\draw [line width=1.3pt,dash pattern=on 2pt off 3pt,color=qqqqff] (16.4,6)-- (16.4,0);

\draw [line width=1.2pt, color=qqqqff] (2.4,5.2) .. controls (7,2.7) and (14,5.3) .. node[pos=0.72, above] {$\Pi^{\alpha}(\pmb{\tau}^{\infty})$} (16.4,2.6);

%feketek
\draw [line width=2pt] (0,0)-- (20,0);
\draw [line width=2pt] (0,6)-- (20,6);
\draw [line width=2pt] (4,-0.15)-- (4,0.15);
\draw (4,-0.15) node[anchor=north] {$\Pi_{\mathcal G}(\mathbf i) = \Pi_{\mathcal G}(\mathbf j)$};
\draw [line width=2pt] (14,0.15)-- (14,-0.15);
\draw (14,-0.15) node[anchor=north] {$\Pi_{\widetilde{\mathcal G}}(\mathbf i)$};
\draw [line width=2pt] (17.6,0.15)-- (17.6,-0.15);
\draw (17.6,-0.15) node[anchor=north] {$\Pi_{\widetilde{\mathcal G}}(\mathbf j)$};
\draw [line width=2pt] (10,0.15)-- (10,-0.15);
\draw (10,-0.15) node[anchor=north] {$\widetilde{x}$};

\draw (10,6) node[anchor=south] {$i(\alpha^*)=j(\alpha^*)$};
\draw [line width=1.3pt,dash pattern=on 2pt off 3pt] (10,6)-- (10,0);
\filldraw (10,4) circle [radius=6pt];
\end{tikzpicture}
\caption{Finding the common fixed point $\widetilde{x}$ at $\alpha^\ast$.}
\label{fig:CommonFixPt}
\end{figure}
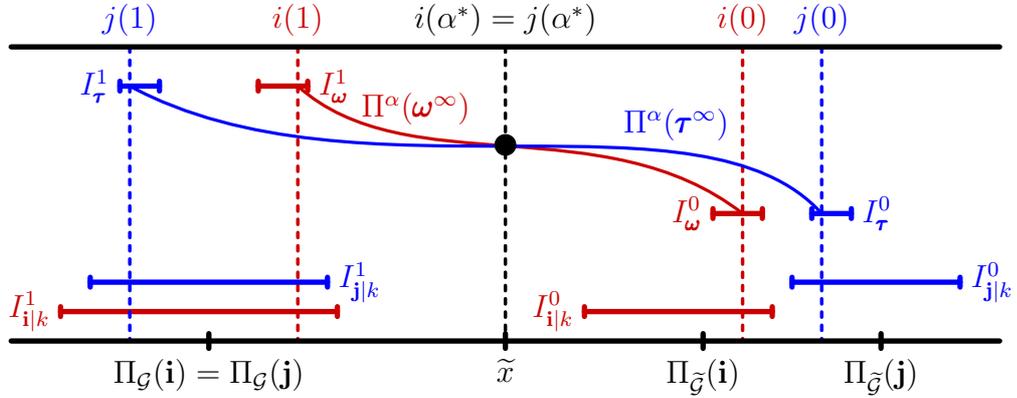

We denote the interval spanned by the attractor $\Lambda^{\alpha}$ of $\mathcal{S}^\alpha$ by $I^\alpha$ and the natural projection by $\Pi^{\alpha}$. Moreover,	for a $\mathbf{u}\in\Sigma^{\ast}$ and $\alpha\in[0,1]$ we write
$$I^\alpha_{\mathbf{u}}:=S^\alpha_{\mathbf{u}}(I^\alpha).$$
The left and right endpoint of $I^\alpha_{\mathbf{u}}$ is denoted by $I^{\alpha,-}_{\mathbf{u}}$ and  $I^{\alpha,+}_{\mathbf{u}}$, respectively. Without loss of generality we may assume that $\Pi_{\widetilde{\mathcal{G}}}(\ii) < \Pi_{\widetilde{\mathcal{G}}}(\jj)$. Hence, we can choose $k$ so large that
\begin{equation}\label{eeq:76}
I_{\ii|k}^{0,+} < I_{\jj|k}^{0,-},
\end{equation}
see Figure~\ref{fig:CommonFixPt}. Since $\Pi_{\mathcal{G}}(\mathbf{i})=\Pi_{\mathcal{G}}(\mathbf{j})$, we have that $I_{\jj|k}^{1,-} \leq I_{\ii|k}^{1,+}$ (otherwise $I_{\jj|k}^{1} \cap I_{\ii|k}^{1}=\emptyset$).
First assume that
\begin{equation}\label{eeq:78}
I_{\jj|k}^{1,-} < I_{\ii|k}^{1,+}.
\end{equation}
Then we can choose $\pmb{\omega},\pmb{\tau}\in\Sigma^{\ast}$ such that $\pmb{\omega}|k=\ii|k$ and $\pmb{\tau}|k=\jj|k$ satisfying
\begin{equation}\label{eeq:75}
I_{\pmb{\tau}}^{1,+}<I_{\pmb{\omega}}^{1,-}	\mbox{ and } \omega_{|\pmb{\omega}|}\neq \tau_{|\pmb{\tau}|}.
\end{equation}
On the other hand, \eqref{eeq:76} implies that $I_{\pmb{\omega}}^{0,+}<I_{\pmb{\tau}}^{0,-}$. Combining this with \eqref{eeq:75} yields that
\begin{equation}\label{eeq:74}
i(0)<j(0) \mbox{ and } i(1)>j(1),
\end{equation}
where $i(\alpha):=\Pi^\alpha(\pmb{\omega}^{\infty })\in I_{\pmb{\omega}}^{\alpha}$ and $j(\alpha):=\Pi^\alpha(\pmb{\tau}^{\infty })\in I_{\pmb{\tau}}^{\alpha}$ denotes the fixed point of $S_{\pmb{\omega}}^{\alpha}$ and $S_{\pmb{\tau}}^{\alpha}$, respectively.

It follows from the definition of $\mathcal{S}^\alpha$ that for every $\mathbf{u}\in\Sigma$ the function $\alpha\mapsto \Pi^\alpha(\mathbf{u})$ is continuous. In particular, for $\pmb{\omega}^{\infty }$ and $\pmb{\tau}^{\infty }$ this and \eqref{eeq:74} implies that there is an $\alpha^\ast\in[0,1]$ such that
\begin{equation*}
\widetilde{x}:=\Pi^{\alpha^\ast}(\pmb{\omega}^{\infty })=  i(\alpha^\ast)=j(\alpha^\ast)=\Pi^{\alpha^\ast}(\pmb{\tau}^{\infty }).
\end{equation*}
Thus, $\widetilde{x}$ is a common fixed point of $S_{\pmb{\omega}}^{\alpha^\ast}$ and $  S_{\pmb{\tau}}^{\alpha^\ast}$. Moreover, with the last part of \eqref{eeq:75}, we showed that $\mathcal{S}^{\alpha^\ast}\in\mathscr{U}$ provided \eqref{eeq:78} holds (recall \eqref{eeq:82}).

If $I_{\ii|k}^{1,+} = I_{\jj|k}^{1,-}$, then either there exists $\ell<k$ such that~\eqref{eeq:78} holds with $I_{\ii|\ell}^{1,+} < I_{\jj|\ell}^{1,-}$ or for every $\ell<k$ we have $I_{\ii|\ell}^{1,+} = I_{\jj|\ell}^{1,-}$. In particular, $I_{i_1}^{1,+} = I_{j_1}^{1,-}$. In this case we choose $\widetilde{\mathcal G}$ from Fact~\ref{fact:96} so that $\Pi_{\mathcal{\widetilde{G}}}(\ii)>\Pi_{\mathcal{\widetilde{ G}}}(\jj)$ and the same argument applies.
\end{proof}

\begin{proof}[Proof of Fact~\ref{fact:98}]
By assumption $\mathcal{S}\in\mathscr{U}$, i.e. there exist $\pmb{\omega}=(\omega_1,\ldots,\omega_k)$ and $\pmb{\tau}=(\tau_1,\ldots,\tau_\ell)\in\Sigma^*$ with $\omega_1\neq\tau_1,\, \omega_{k}\neq
\tau_{\ell}$ and $\widetilde{x}\in X$ such that $$\widetilde{x}=S_{\pmb{\omega}}(\widetilde{x})
=S_{\pmb{\tau}}(\widetilde{x}).$$

For $p\in\left\{2, \dots ,k\right\}$ and $q \in\left\{2, \dots ,\ell \right\}$ let
	$$
	y_p:=S_{\omega_{p} \dots \omega_k}(\widetilde{x}) \mbox{ and }
	z_q:=S_{\tau_q \dots \tau_\ell }(\widetilde{x}).
	$$
Using the function
	$$
	L(x):=\prod\limits_{p=2}^{k}
	\left(x-y_p\right)^2 \cdot
	\prod\limits_{q=2}^{\ell }
	\left(x-z_q\right)^2,
	$$
we define the IFS $\widetilde{\mathcal{S}}=\big\{\widetilde{S}_1,\ldots,\widetilde{S}_m\big\}$ to be
	$$
	\widetilde{S}_i(x):=
	\left\{
	\begin{array}{ll}
	S_i(x)+\varepsilon \cdot  L(x) \cdot (x-\widetilde{x})^2, & \hbox{if $i\neq \omega_k$;} \\
	S_i(x)+\varepsilon \cdot  L(x) \cdot (x-\widetilde{x}), & \hbox{if $i=\omega_k$.}
	\end{array}
	\right.
	$$
Then $\widetilde{x}$ is still a common fixed point:
	\begin{equation}\label{eeq:90}
	\widetilde{x}=\widetilde{S}_{\pmb{\omega}}(\widetilde{x})
	=\widetilde{S}_{\pmb{\tau}}(\widetilde{x}).
	\end{equation}
Since $L(y_p) = 0 = L(z_q)$ for every $p\in\left\{2, \dots ,k\right\}$ and $q \in\left\{2, \dots ,\ell \right\}$, a simple induction shows that
\begin{equation}\label{eeq:88}
\widetilde{S}_{\omega_{p} \dots \omega_k}(\widetilde{x}) = y_p \,\mbox{ and }\,
\widetilde{S}_{\tau_q \dots \tau_\ell }(\widetilde{x}) = z_q.
\end{equation}
As for the derivatives
	\begin{equation*}
	\big(\widetilde{S}_i\big)'(x)
	=
	\left\{
	\begin{array}{ll}
	S'_i(x)+\varepsilon L'(x)(x-\widetilde{x})^2+2\varepsilon L(x)(x-\widetilde{x}), & \hbox{if $i\neq \omega_k$;} \\
	S'_i(x)+\varepsilon L'(x)(x-\widetilde{x})+\varepsilon L(x), & \hbox{if $i=\omega_k$.}
	\end{array}
	\right.
	\end{equation*}
Similarly, $L'(y_p)=0=L'(z_q)$ for all $p$ and $q$ implies that for all $i\in\{1,\ldots,m\}$
	\begin{equation}\label{eeq:87}
	\big(\widetilde{S}_i\big)'(t)=S'_i(t) \,\mbox{ if } t\in\left\{
	y_2, \dots ,y_k,z_2, \dots ,z_{\ell}
	\right\},
	\end{equation}
furthermore, $L(\widetilde{x})\neq 0$ implies that
	\begin{equation}\label{eeq:86}
	\big(\widetilde{S}_i\big)'(\widetilde{x})=
	\left\{
	\begin{array}{ll}
	S'_i(\widetilde{x}), & \hbox{if $i\neq \omega_k;$} \\
	S'_i(\widetilde{x})+\varepsilon \cdot  L(\widetilde{x}), & \hbox{if $i=\omega_k.$}
	\end{array}
	\right.
	\end{equation}
Hence, it follows from \eqref{eeq:88}, \eqref{eeq:87} and \eqref{eeq:86} that
	\begin{equation*}\label{eeq:85}
	\big(\widetilde{S}_{\pmb{\omega}}\big)'(\widetilde{x})
	=
	S'_{\pmb{\omega}}(\widetilde{x})
	+
	\varepsilon \cdot L(\widetilde{x}) \cdot S'_{\pmb{\omega}^-}(y_k)
	\,\mbox{ and }\,
	\big(\widetilde{S}_{\pmb{\tau}}\big)'(\widetilde{x}) =	S'_{\pmb{\tau}}(\widetilde{x}).
	\end{equation*}
Choosing $\varepsilon$ small enough, $\mathrm{dist}\big(\mathcal{S},
\widetilde{\mathcal{S}}\big)$ can be made arbitrarily small (in the metric defined in \eqref{eq:38}) and at the same time $\log| \widetilde{S}'_{\pmb{\omega}}(\widetilde{x})|/ \log|\widetilde{S}'_{\pmb{\tau}}(\widetilde{x})|\not\in\mathbb{Q}$.
This together with \eqref{eeq:90} shows that $\widetilde{S}\in\mathscr{R}$.
\end{proof}

\begin{proof}[Proof of Fact~\ref{fact:99}]
	First, let us recall the definition of $\Phi(x,r)$ from \eqref{eq:93}. Let us generalise it as follows: for $N\geq1$, let
	\begin{equation*}\label{eq:93K}
	\Phi_N(x,r):=
	\big\{
	S_{\iiv}|_\Lambda:
	\mathrm{diam}(S_{\iiv}(\Lambda))
	\leq
	r
	<
	\mathrm{diam}(S_{\iiv^{-N}}(\Lambda)),
%	\\
	S_{\iiv}(\Lambda)\cap B(x,r)\neq  \emptyset,\ \iiv\in\Sigma^{\ast}
	\big\},
	\end{equation*}
where $\iiv^{-N}$ is the prefix of $\iiv$ by removing the last $N$ symbols. Recall from~\eqref{eq:92} that the WSP holds if $\sup_{x\in\Lambda,r>0}\Phi(x,r)<\infty$, and thus, WSP implies that for every $N\geq1$, $\sup_{x\in\Lambda,r>0}\Phi_N(x,r)<\infty$. Indeed, if $f|_\Lambda\in \Phi_N(x,R)$ then for every finite word $\iiv\in\Sigma^*$ for which $S_{\iiv}|_\Lambda=f|_\Lambda$, there exists a prefix $\jjv$ of $\iiv$ such that
$$
\mathrm{diam}(S_{\iiv}(\Lambda))\leq\mathrm{diam}(S_{\jjv}(\Lambda))
\leq
r
<\mathrm{diam}(S_{\jjv^{-}}(\Lambda))\leq\mathrm{diam}(S_{\iiv^{-N}}(\Lambda)).
$$
Moreover, $S_{\iiv}(\Lambda)\cap B(x,r)\neq  \emptyset$ clearly implies that $S_{\jjv}(\Lambda)\cap B(x,r)\neq  \emptyset$. Thus, $S_{\jjv}|_\Lambda\in\Phi(x,r)$. Hence, for every $g|_\Lambda\in \Phi(x,r)$ there exist at most $1+m+\cdots+m^{N-1}$-many maps $f|_\Lambda\in \Phi_N(x,r)$ such that $f\circ S_{\hbar}|_\Lambda=g|_\Lambda$ for some $\hbar$ with $|\hbar|\leq K$, and so, $\Phi_N(x,r)\leq (m^{N}-1)/(m-1)\Phi(x,r)$ for every $x\in\Lambda$ and $r>0$.
	
Now, let $\mathcal{S}\in\mathscr{R}$ (introduced in \eqref{cv83}), i.e. there exist
$ \widetilde{x}\in X,
\, \pmb{\omega},\pmb{\tau}\in\Sigma^*$ for which $\widetilde{x}=S_{\pmb{\omega}}(\widetilde{x})
=S_{\pmb{\tau}}(\widetilde{x})$ and $\log |S'_{\pmb{\omega}}(\widetilde{x})| / \log |S'_{\pmb{\tau}}(\widetilde{x})| \not\in\mathbb{Q}$. Observe that $\widetilde{x}\in \Lambda$, since $\widetilde{x}=\Pi(\pmb{\omega}^\infty)$. We claim that $\Phi_N(\widetilde{x},R)$ can be made arbitrarily large for an appropriately chosen $R$, where $N=\max\{|\pmb{\omega}|,|\pmb{\tau}|\}$, implying that $\mathcal{S}\in\mathscr{NOWSP}$ by the discussion above.

For brevity, let us write
$$
g_1:=S_{\pmb{\omega}} \,\mbox{ and }\, g_2:=S_{\pmb{\tau}},
\qquad
a:=|g'_1(\widetilde{x})| \,\mbox{ and }\, b:=|g'_2(\widetilde{x})|.
$$
Since $\log a / \log b \not\in\mathbb{Q}$, Dirichlet's approximation theorem implies that there are infinitely many $j\in\mathbb{N}$ for which we can find $i\in\mathbb{N}$ such that
\begin{equation*}\label{eq:24}
 0< \left|
  \frac{\log a}{\log b}-\frac{i}{j}
  \right|<\frac{1}{j^2}.
\end{equation*}
After rearranging
\begin{equation}\label{eq:25}
 b^{1/j}<
 \frac{a^j}{b^i}
 <b^{-1/j} \,\mbox{ and }\, a^j\neq b^i.
\end{equation}
Let us fix such an $i,j$ pair. For $0\leq r<\sqrt{j}$ we introduce
\begin{equation}\label{cv81}
h_r:=\underbrace{g_{1}\circ\cdots\circ g_{1}}_{j\cdot(\lfloor\sqrt{j}\rfloor-r)}\circ
\underbrace{g_{2}\circ\cdots\circ g_{2}}_{i\cdot r}.
\end{equation}
Since $\widetilde{x}=g_1(\widetilde{x})=g_2(\widetilde{x})$, we get that $|h'_r(\widetilde{x})|=a^{j\cdot(\lfloor\sqrt{j}\rfloor-r)} b^{i\cdot r}$. Hence, \eqref{eq:25} implies that for every $0 \leq r_1<r_2<\sqrt{j}$
\begin{equation*}\label{eq:27}
  \frac{|h'_{r_1}(\widetilde{x})|}{|h'_{r_2}(\widetilde{x})|}
  =\left(\frac{a^j}{b^i}\right)^{r_2-r_1}
  \in
  \left(
  b^{(r_2-r_1)/j},b^{-(r_2-r_1)/j}
  \right)
   \subset
   \left(
   b^{1/\sqrt{j}},b^{-1/\sqrt{j}}
   \right).
\end{equation*}

By the Bounded Distortion Property there is a constant $C>0$ independent on $j$ such that for all $0 \leq r_1<r_2<\sqrt{j}$:
\begin{equation*}\label{eq:28}
  \frac{\mathrm{diam}(h_{r_1}(\Lambda))}
  {\mathrm{diam}(h_{r_2}(\Lambda))}\in
  \left(
  C^{-1}b^{1/\sqrt{j}},Cb^{-1/\sqrt{j}}
  \right) \subset \left(\frac{1}{2C},2C\right), \mbox{ for large } j.
\end{equation*}
Let $K:=\big\lceil\frac{\log 4C}{\log(1/\tau
)}\big\rceil$. We can construct $K$ open intervals $\left\{(w_k,z_k)\right\}_{k=1}^{K}$ such that
\begin{equation*}\label{cv99}
  \left(\frac{1}{2C},2C\right) \subset \bigcup\limits_{k =1}^{K}
  (w_k,z_k) \mbox{ and }  z_k<\tau \cdot w_k, \quad k=1, \dots ,K,
\end{equation*}
where $\tau$ was defined in Claim \ref{cv93}.
Then there exists an $\ell \leq K$ such that for
\begin{equation}\label{cv95}
\mathcal{I}:=\Big\{(r_1,r_2):
0 \leq r_1<r_2<\sqrt{j},\
\frac{\mathrm{diam}(h_{r_1}(\Lambda))}
  {\mathrm{diam}(h_{r_2}(\Lambda))}
\in (w_{\ell },z_{\ell })
\Big\},
\end{equation}
we have $\#\mathcal{I} \geq
 \binom{\lfloor\sqrt{j}\rfloor+1}{2}/K$. That is if $j $
 is large enough then
 \begin{equation}\label{cv98}
   \#\mathcal{I} \geq\frac{j}{4K}.
 \end{equation}
Now we partition the pairs contained in $\mathcal{I}$ according to their second components. That is for every $r\in\left\{0,1, \dots , \lfloor\sqrt{j}\rfloor\right\}$, we introduce the disjoint sets
$$
\mathcal{I}_r:=
\left\{r_1 :
(r_1,r)\in\mathcal{I}
\right\}.
$$
By definition $\mathcal{I}_r \subset \left\{0, \dots , \lfloor\sqrt{j}\rfloor \right\}$.
So, by \eqref{cv98} we can
fix an $r_*\in \left\{1, \dots ,\lfloor\sqrt{j}\rfloor\right\}$ such that
\begin{equation}\label{cv97}
  \#\mathcal{I}_{r_*} \geq \frac{\sqrt{j}}{4K}.
\end{equation}
We choose an $\widehat{r}\in \mathcal{I}_{r_*}$ such that
\begin{equation*}\label{cv82}
  \mathrm{diam}(h_{\widehat{r}}(\Lambda))=
  \max\left\{ \mathrm{diam}(h_{r_1}(\Lambda)):r_1\in\mathcal{I}_{r_*}\right\}
  =:
  \eta.
\end{equation*}
Observe that for every $r_1\in\mathcal{I}_{r_*}$ we have
\begin{equation}\label{cv96}
  \frac{1}{\tau} \cdot \eta \leq \mathrm{diam}(h_{r_1}(\Lambda)) \leq \eta
  = \mathrm{diam}(h_{\widehat{r}}(\Lambda)).
\end{equation}
Indeed,  let $r_1\in\mathcal{I}_{r_*}$ be arbitrary. Then $(r_1,r_*)\in\mathcal{I}$. By the definition \eqref{cv95} of $\mathcal{I}$ we have
$$
\frac{\mathrm{diam}(h_{r_1}(\Lambda))}
  {\mathrm{diam}(h_{r_*}(\Lambda))}
\in (w_{\ell },z_{\ell })
\;\mbox{ and also }\;
\frac{\mathrm{diam}(h_{\widehat{r}}(\Lambda))}
  {\mathrm{diam}(h_{r_*}(\Lambda))}
\in (w_{\ell },z_{\ell }).
$$
Using this and \eqref{cv96} we get
\begin{equation}\label{cv87}
 \frac{1}{\tau} \leq \frac{w_\ell }{z_\ell } \leq
\frac{\mathrm{diam}(h_{r_1}(\Lambda))}
  {\eta}
 \leq 1,\quad \forall r_1\in\mathcal{I}_{r_*}.
\end{equation}

We introduce $h_{r^-}$ for an $r\in \left\{0, \dots ,\lfloor\sqrt{j}\rfloor\right\}$ as follows: if $r \geq 1$ then
$$h_{r^-}:=\underbrace{g_{1}\circ\cdots\circ g_{1}}_{j\cdot(\lfloor\sqrt{j}\rfloor-r)}\circ
\underbrace{g_{2}\circ\cdots\circ g_{2}}_{i\cdot r-1}.$$ If
$r=0$ then
$h_{r^-}:=\underbrace{g_{1}\circ\cdots\circ g_{1}}_{j\cdot\lfloor\sqrt{j}\rfloor-1}$ (cf. \eqref{cv81}).
Then \eqref{cv87} and Claim \ref{cv93} imply that
\begin{equation}\label{cv86}
 \mathrm{diam}(h_{r_1}(\Lambda))  \leq \eta \;\mbox{ and }\;
  \mathrm{diam}(h_{r_1^-}(\Lambda))>\eta,\quad \forall r_1\in\mathcal{I}_{r_*}.
\end{equation}
By the definition of the mapping  $h_r,$ we get that
 for all $0\leq r_1<\sqrt{j}$
\begin{equation}\label{cv85}
  \widetilde{x}=h_{r_1}(\widetilde{x})\in h_{r_1}(\Lambda).
\end{equation}
The combination of  \eqref{cv86} and \eqref{cv85} yields that
\begin{equation}\label{cv84}
  \mathcal{I}_{r_*} \subset \Phi_N\left(\widetilde{x},\eta\right).
\end{equation}
Since $j$ can be arbitrarily large, \eqref{cv97} and \eqref{cv84} together imply that the WSP does not hold.
\end{proof}

%\begin{comment}

%%%%%%%%%%%%%%%%%%%%%%%%%%%%%%%%%%%%%%%%%%%%%%%%%%%%%%%%%%%%%%%%%%%%%%%%%%%
%%%%%%%%%%%%%%%%%%%%%%%%%%%%%%%%%%%%%%%%%%%%%%%%%%%%%%%%%%%%%%%%%%%%%%%%%%%
\section{Transversal families, proof of Theorem~\ref{thm:99}}\label{sec:05}

Let $\{\mathcal{S}^{\underline{\lambda}}: \underline{\lambda} \in B\}$ be a transversal family of self-conformal IFSs, recall Definition~\ref{def:98}. We defined the sets $\mathscr{SSP}, \mathscr{WSP}$ and $\mathscr{EO}$ as those $\underline{\lambda} \in B$ for which $\mathcal{S}^{\underline{\lambda}}$ satisfies the SSP, the WSP or has an exact overlap. Let us also introduce
\begin{equation*}%\label{eq:86}
\mathscr{OSC} :=
\big\{
\underline{\lambda} \in B:
\mathcal{S}^{\underline{\lambda} } \mbox{ satisfies the OSC}
\big\}.
\end{equation*}
We prove the assertions of Theorem~\ref{thm:99} in separate propositions.
\begin{comment}
Recall that a subset $H \subset \mathbb{R}^d$ is a very small set if
\begin{equation*}\label{eq:52}
\mathcal{L}_d(H)=0 \;\mbox{ and }\; H
\mbox{ is a set of first category}.
\end{equation*}
We will use without mentioning the trivial fact that a countable union of very small sets is very small.

Notice that $\mathscr{SSP}\cap \mathscr{EO}=\emptyset$. Hence, in light of Fact~\ref{eq:31}, to prove Theorem~\ref{thm:99} it is enough to show that
\begin{equation*}
\mathscr{EO} \text{ and } (B\setminus\mathscr{SSP})\cap\mathscr{OSC} \text{ are both very small sets.}
\end{equation*}
\end{comment}

\begin{prop}\label{lem:98}
Let $B \subset \mathbb{R}^d$ be a non-degenerate ball and let $\{\mathcal{S}^{\underline{\lambda}}: \underline{\lambda} \in B\}$ be a transversal family of self-conformal IFSs on the line. Then

\begin{equation*}
\mathscr{EO} \text{ is a set of first category and } \mathcal{L}_d(\mathscr{EO}) = 0.
\end{equation*}

\end{prop}
\begin{proof}
We can write $\mathscr{EO}$ as the countable union
\begin{equation*}
\mathscr{EO} = \bigcup_{\stackrel{\iiv,\jjv\in\Sigma^\ast}{i_1\neq j_1}} B_{\iiv,\jjv},
\end{equation*}
where $B_{\iiv,\jjv}= \big\{\underline{\lambda}\in B: S_{\iiv}^{\underline{\lambda} }\equiv S_{\jjv}^{\underline{\lambda} } \big\}$. Hence, it is enough to show that each $B_{\iiv,\jjv}$ is a set of first category and $\mathcal{L}_d(B_{\iiv,\jjv}) = 0$. We set
$$
\mathbf{i}:=\iiv\mathbf{1} \mbox{ and }
\mathbf{j}:=\jjv\mathbf{1}.
$$
Using that $\Pi_{\underline{\lambda} }(\mathbf{i})=
S_{\iiv}^{\underline{\lambda} }(\Pi_{\underline{\lambda} }(\mathbf{1}))
$ and
$\Pi_{\underline{\lambda} }(\mathbf{j})=
S_{\jjv}^{\underline{\lambda} }(\Pi_{\underline{\lambda} }(\mathbf{1}))
$ we get that
\begin{equation*}%\label{eq:50}
B_{\iiv,\jjv} \subset
\widetilde{B}_{\iiv,\jjv}:
=
\left\{\underline{\lambda} \in B
:
\Pi_{\underline{\lambda} }(\mathbf{i})-
\Pi_{\underline{\lambda} }(\mathbf{j})=0
\right\}.
\end{equation*}
We claim that
\begin{equation}\label{eq:53}
\widetilde{B}_{\iiv,\jjv} \mbox{ is a set of first category and } \mathcal{L}_d(\widetilde{B}_{\iiv,\jjv}) = 0,
\end{equation}
which implies the assertion of the proposition.

To show \eqref{eq:53}, we fix a $\underline{\lambda}'\in \widetilde{B}_{\iiv,\jjv}\cap \mathrm{int}(B)$ if $B_{\iiv,\jjv}\neq \emptyset $ and denote $f:\R^d\to\R$
$$f(\underline{\lambda} ):= \Pi_{\underline{\lambda} }(\mathbf{i})-
\Pi_{\underline{\lambda} }(\mathbf{j}).$$
Since $f(\underline{\lambda}')=0$ and $\{\mathcal{S}^{\underline{\lambda}}: \underline{\lambda} \in B\}$ is a transversal family, the transversality condition \eqref{eq:15} implies that one of the coordinates of $\nabla_{\underline{\lambda}} f(\underline{\lambda}')$ is positive in absolute value. Without loss of generality we may assume that this is the last coordinate:
\begin{equation*}%\label{eq:49}
\left|
\frac{\partial f}{\partial \lambda_d}(\underline{\lambda}')
\right|>\zeta>0.
\end{equation*}
By assumption, $\underline{\lambda} \mapsto
\Pi_{\underline{\lambda} }(\mathbf{i})$ is continuously differentiable, recall~\eqref{eq:42}, so there is a neighborhood $M$ of $\underline{\lambda} '$ such that
\begin{equation}\label{eq:47}
\left|
\frac{\partial f}{\partial \lambda_d}(\underline{\lambda})
\right|>0 \mbox{ for } \underline{\lambda} \in M.
\end{equation}
Let $\mathrm{proj}$ be the projection to the first $d-1$
coordinates:
$$
\mathrm{proj}(\lambda_1, \dots ,\lambda_d):=
(\lambda_1, \dots ,\lambda_{d-1}).
$$
We write
$$
B^{\ast}:=\mathrm{proj}(B),\quad
\underline{\lambda} ^{\ast}:=\mathrm{proj}(\underline{\lambda} ),
\;\mbox{ and }\;
{\underline{\lambda}'} ^{\ast}:=\mathrm{proj}(\underline{\lambda}').
$$
The Implicit Function Theorem implies that there exists an open neighborhood
$N \subset \mathrm{int}(B^{\ast})$ of ${\underline{\lambda}' }^{\ast}$ and there exists a unique continuously differentiable function $g:N\to \mathbb{R}$ such that $g({\underline{\lambda}'} ^{\ast})=\lambda'_d$ and for all $\underline{\lambda} ^{\ast} \in N$ we have
 \begin{equation*}%\label{eq:48}
   (\underline{\lambda} ^{\ast} ,g(\underline{\lambda}  ^{\ast} ))\in M  \;\mbox{ and }\;
   f(\underline{\lambda} ^{\ast},g(\underline{\lambda} ^{\ast}))=0.
 \end{equation*}
 Then it follows from \eqref{eq:47} that
 \begin{equation*}%\label{eq:46}
   \widetilde{B}_{\iiv,\jjv}\cap M=\left\{
   (\underline{\lambda} ^{\ast},g(\underline{\lambda} ^{\ast})):\,\underline{\lambda} ^{\ast}\in N
   \right\}.
 \end{equation*}
From the fact that $g$ is continuously differentiable we obtain that the set on the right hand-side is a set of first category and has zero $d$-dimensional Lebesgue measure. Then using a usual compactness argument we obtain the same for
$\widetilde{B}_{\iiv,\jjv}$. This completes the proof of \eqref{eq:53}.
\end{proof}

In light of Fact~\ref{eq:31}, to prove the claims in Theorem~\ref{thm:99} for $(B\setminus\mathscr{SSP})\cap\mathscr{WSP}$ it is enough to show the same for $(B\setminus\mathscr{SSP})\cap\mathscr{OSC}$.

\begin{prop}\label{lem:97}
Let $B \subset \mathbb{R}^d$ be a non-degenerate closed ball and let $\{\mathcal{S}^{\underline{\lambda}}: \underline{\lambda} \in B\}$ be a transversal family of self-conformal IFSs on the line. Then
\begin{equation*}
(B\setminus\mathscr{SSP})\cap\mathscr{OSC} \text{ is a set of first category.}
\end{equation*}

\end{prop}
The proof of Proposition~\ref{lem:97} relies on the following auxiliary lemma.
\begin{lemma}\label{lem:51}
Let $\mathcal{F}:=\{f_1,f_2,f_3\}\in \Theta_{\beta,\rho}^3([0,1])$. Assume further that
\begin{enumerate}
\item $f_1(0)=f_2(0)=0$,
\item $f'_1(0)=f'_2(0)=\alpha$, and
\item $f_3(1)=1$.
\end{enumerate}
Then $\mathcal{F}$ does not satisfy the OSC.
\end{lemma}
\begin{proof}
The proof goes by contraposition. Assume that there exists an open bounded set $U$ such that for every $i\neq j\in\{1,2,3\}$
\begin{equation*}
f_i(U)\subset U \;\;\text{and}\;\; f_i(U)\cap f_j(U)=\emptyset.
\end{equation*}
Let $I$ be an open interval in $U$ and denote $J:=f_3(I)=(a_0,b_0)$. Moreover, let $\{1,2\}^{\ast}$ denote all finite length words with entries either $1$ or $2$.

Writing $a_{\iiv}:=f_{\iiv}(a_0)$ and $b_{\iiv}:=f_{\iiv}(b_0)$ for $\iiv\in\{1,2\}^{\ast}$, we have that
\begin{equation}\label{eq:500}
C_0^{-1}\alpha^n a_0 \leq a_{\iiv} \leq C_0 \alpha^n a_0 \;\;\text{and}\;\; C_0^{-1}\alpha^n b_0 \leq b_{\iiv} \leq C_0 \alpha^n b_0,
\end{equation}
where $|\iiv|=n$ and $C_0$ is the constant from the bounded distortion property~\eqref{eq:96}. Indeed, using that $f_{\iiv}(0)=0$, the mean value theorem, the bounded distortion property and $(2)$ implies that
\begin{equation*}
a_{\iiv} = f_{\iiv}(a_0) = f'_{\iiv}(\xi) a_0\leq C_0 f'_{\iiv}(0)a_0=C_0\alpha^n a_0.
\end{equation*}
The argument for the other direction and $b_{\iiv}$ is exactly the same.

Hence, \eqref{eq:500} implies that for every $\iiv\in\{1,2\}^{\ast}$
\begin{equation}\label{eq:501}
0<C_0^{-2}\, \frac{a_0}{b_0} \leq \frac{a_{\iiv}}{b_{\iiv}} \leq C_0^2\, \frac{a_0}{b_0} < \infty.
\end{equation}
Moreover, for every $i=1,2$
\begin{equation}\label{eq:502}
\frac{f_i(a_{\iiv})}{a_{\iiv}} = \frac{f_i(f_{\iiv}(a_0)) - f_i(f_{\iiv}(0))}{a_{\iiv}}\to f'_i(0)=\alpha \;\;\text{as } |\iiv|\to\infty,
\end{equation}
since $f_{\iiv}(a_0)\to0$. The same limit holds for $f_i(b_{\iiv})/b_{\iiv}$.

We also claim that for every $\iiv\neq \jjv\in\{1,2\}^{\ast}$
\begin{equation*}
f_{\iiv}(J)\cap f_{\jjv}(J)=\emptyset.
\end{equation*}
Indeed, there is a unique $\underline{\omega}$ and $\underline{\tau}$ with $\omega_1\neq \tau_1$ for which $\iiv3=(\iiv\wedge\jjv)\underline{\omega}$ and $\jjv3=(\iiv\wedge\jjv)\underline{\tau}$, moreover, $f_{\omega_1}(U)\cap f_{\tau_1}(U) = \emptyset$. Then
\begin{equation*}
f_{\iiv}(J)\cap f_{\jjv}(J)=f_{\iiv\wedge\jjv}\big( f_{\underline{\omega}}(I)\cap f_{\underline{\tau}}(I) \big) \subseteq f_{\iiv\wedge\jjv}\big( f_{\omega_1}(U)\cap f_{\tau_1}(U) \big) = \emptyset.
\end{equation*}
Therefore, $(a_{1\iiv},b_{1\iiv})\cap (a_{2\iiv},b_{2\iiv})=\emptyset$ and we may assume without loss of generality that $b_{1\iiv}\leq a_{2\iiv}$ for infinitely many $\iiv\in\{1,2\}^{\ast}$.

\noindent\textbf{Case I.} Of these $\iiv$, choose a sequence such that $a_{\iiv}/b_{\iiv}\to z\neq 1$. Then
\begin{equation*}
\frac{b_{2\iiv}-a_{2\iiv}}{b_{2\iiv}-a_{1\iiv}} = \frac{ \frac{f_2(b_{\iiv})}{b_{\iiv}}-\frac{f_2(a_{\iiv})}{a_{\iiv}}\cdot\frac{a_{\iiv}}{b_{\iiv}} }{ \frac{f_2(b_{\iiv})}{b_{\iiv}}-\frac{f_1(a_{\iiv})}{a_{\iiv}}\cdot\frac{a_{\iiv}}{b_{\iiv}} } \to \frac{\alpha-\alpha z}{\alpha-\alpha z} =1,
\end{equation*}
where we used \eqref{eq:501} and \eqref{eq:502}. Similarly,
\begin{equation*}
\frac{b_{1\iiv}-a_{1\iiv}}{b_{2\iiv}-a_{1\iiv}}\to \frac{\alpha-\alpha z}{\alpha-\alpha z} =1.
\end{equation*}
Thus, $a_{2\iiv}<b_{1\iiv}$ for $|\iiv|$ sufficiently large, contradicting that $b_{1\iiv}\leq a_{2\iiv}$.

\noindent\textbf{Case II.} Now assume $a_{\iiv}/b_{\iiv}\to 1$ as $|\iiv|\to\infty$. Then a combination of the bounded distortion property, the mean value theorem, and \eqref{eq:500} yields
\begin{equation*}
C_0^{-1} \alpha^n (b_0-a_0) \leq f'_{\iiv}(\xi) (b_0-a_0) = b_{\iiv}-a_{\iiv} = b_{\iiv} \Big( 1-\frac{a_{\iiv}}{b_{\iiv}}\Big) \leq C_0 \alpha^n b_0 \Big( 1-\frac{a_{\iiv}}{b_{\iiv}}\Big).
\end{equation*}
As a result $1-a_{\iiv}/b_{\iiv}\geq C_0^{-2}b_0^{-1}(b_0-a_0)>0$, which contradicts $a_{\iiv}/b_{\iiv}\to 1$.
\end{proof}

%\begin{claim}\label{claim:50}
%If $\{\mathcal{S}^{\underline{\lambda}}: \underline{\lambda} \in B\}$ is a transversal family of self-conformal IFSs on the line, then $(B\setminus\mathscr{SSP})\cap\mathscr{OSC}$ is a set of first category.
%\end{claim}

\begin{proof}[Proof of Proposition~\ref{lem:97}]
It is enough to prove that the compliment of $\mathscr{OSC}$ is a dense $G_\delta$ set. We first argue that it is dense, i.e. for every $\underline{\lambda}_0\notin\mathscr{SSP}$ and $\varepsilon>0$
\begin{equation}\label{eq:503}
\text{there exits } \underline{\lambda}^{\ast}\in B(\underline{\lambda}_0,\varepsilon) \text{ such that } \underline{\lambda}^{\ast}\notin \mathscr{OSC}.
\end{equation}

Choose $\underline{\lambda}_0\notin\mathscr{SSP}$, i.e. there exist $\ii,\jj\in\Sigma$ with $i_1\neq j_1$ such that $\Pi_{\underline{\lambda}_0}(\ii)=\Pi_{\underline{\lambda}_0}(\jj)$. We may assume that $\underline{\lambda}_0\in\mathrm{int} B$ because the boundary of $B$ is a first category set with zero $d$-dimensional Lebesgue measure. The transversality condition~\eqref{eq:15} implies that $\big| \frac{\partial}{\partial \lambda_{\ast}}(\Pi_{\underline{\lambda}}(\ii)-\Pi_{\underline{\lambda}}(\jj))|_{\underline{\lambda}=\underline{\lambda}_0}\big|>0$ for some coordinate $\lambda_{\ast}$. Along this direction there exists parameter $\underline{\omega}\in\mathrm{int} B$ arbitrarily close to $\underline{\lambda}_0$ such that $\Pi_{\underline{\omega}}(\ii)\neq \Pi_{\underline{\omega}}(\jj)$. An analogous argument to the one in the proof of Fact~\ref{fact:97} implies that
for every $\varepsilon>0$ there exist $x_0\in X$, $\underline{\lambda}^{\ast}\in B(\underline{\lambda}_0,\varepsilon)$ and $n,m\geq 1$, for which
\begin{equation*}
S^{\underline{\lambda}^{\ast}}_{\ii|n}(x_0) =x_0 = S^{\underline{\lambda}^{\ast}}_{\jj|m}(x_0).
\end{equation*}
\noindent\textbf{Case I.} If
\begin{equation*}
\frac{\log \big| \big(S^{\underline{\lambda}^{\ast}}_{\ii|n}\big)'(x_0) \big|}{\log \big| \big(S^{\underline{\lambda}^{\ast}}_{\jj|m}\big)'(x_0) \big|} \notin \mathbb{Q},
\end{equation*}
then Fact~\ref{fact:99} implies that $\underline{\lambda}^{\ast}\notin \mathscr{WSP}$, in particular, $\underline{\lambda}^{\ast}\notin \mathscr{OSC}$.

\noindent\textbf{Case II.} If
\begin{equation*}
\frac{\log \big| \big(S^{\underline{\lambda}^{\ast}}_{\ii|n}\big)'(x_0) \big|}{\log \big| \big(S^{\underline{\lambda}^{\ast}}_{\jj|m}\big)'(x_0) \big|} =\frac{p}{q}\in \mathbb{Q},
\end{equation*}
then consider the IFS
\begin{equation*}
f_1:= S^{\underline{\lambda}^{\ast}}_{(\ii|n)^{2q}},\; f_2:= S^{\underline{\lambda}^{\ast}}_{(\jj|m)^{2p}} \;\text{ and }\; f_3,
\end{equation*}
where $f_3$ is any other map with fixpoint other than $x_0$ that is the composition of maps from $\mathcal{S}^{\underline{\lambda}^{\ast}}$. Then Lemma~\ref{lem:51} implies that $\{f_1,f_2,f_3\}$ does not satisfy the OSC, hence, $\underline{\lambda}^{\ast}\notin \mathscr{OSC}$. This proves~\eqref{eq:503}.

Now we show that $B\setminus \mathscr{OSC}$ can be expressed as the countable intersection of open sets. Let $s_0(\underline{\lambda})$ denote the conformal dimension of $\mathcal{S}^{\underline{\lambda}}$. Recall the equivalent characterizations of OSC from~\eqref{eq:32}, in particular
\begin{equation*}
B\setminus \mathscr{OSC} = \big\{\underline{\lambda}\in B:\; \mathcal{H}^{s_0(\underline{\lambda})}(\Lambda^{\underline{\lambda}})=0\big\} = \bigcap_{k>0} \underbrace{ \Big\{\underline{\lambda}\in \mathrm{int} B:\; \mathcal{H}^{s_0(\underline{\lambda})}(\Lambda^{\underline{\lambda}})<\frac{1}{k}\Big\} }_{=:\, J_{k}}.
\end{equation*}
Using the definition of Hausdorff measure, $J_{k}$ is equal to the countable intersection
\begin{equation*}
\bigcap_{n>0} \bigg\{\underline{\lambda}\in \mathrm{int} B:\; \exists\, \big\{A^{\underline{\lambda}}_i\big\} \text{ such that } \Lambda^{\underline{\lambda}}\subset \bigcup_i A^{\underline{\lambda}}_i,\, \big| A^{\underline{\lambda}}_i \big|<\frac{1}{n},\, \sum_i\big| A^{\underline{\lambda}}_i \big|^{s_0(\underline{\lambda})}<\frac{1}{k} \bigg\}.
\end{equation*}
Since each of the sets in the intersection is open, the assertion follows.
\end{proof}

\begin{prop}\label{claim:51}
Let $B \subset \mathbb{R}^d$ be a non-degenerate closed ball and let $\{\mathcal{S}^{\underline{\lambda}}\}_{\underline{\lambda} \in B}$ be a transversal family of self-similar IFSs on the line, then $\mathcal{L}_d\big((B\setminus\mathscr{SSP})\cap\mathscr{OSC}\big)=0.$
\end{prop}
\begin{proof}
For $d=1$, the statement was proved in~\cite[Theorem 2.1]{PSS_SelfSIm_2000IsrlJ}, for higher dimensions only an outline of the proof was given, see ~\cite[Theorem 7.1]{PSS_SelfSIm_2000IsrlJ}. For the convenience of the reader, we include a detailed argument in Appendix~\ref{sec:appendix}.
\end{proof}

\begin{proof}[Proof of Theorem~\ref{thm:99}]
	Follows directly from Propositions~\ref{lem:98}, \ref{lem:97} and \ref{claim:51}.
\end{proof}

%%%%%%%%%%%%%%%%%%%%%%%%%%%%%%%%%%%%%%%%%%%%%%%%%%%%%%%%%%%%%%%%%%%%%%%%%%%
%%%%%%%%%%%%%%%%%%%%%%%%%%%%%%%%%%%%%%%%%%%%%%%%%%%%%%%%%%%%%%%%%%%%%%%%%%%
\appendix

\section{Detailed proof of Proposition~\ref{claim:51}}\label{sec:appendix}

In the self-similar case, we use another equivalent characterization of the OSC due to Bandt and Graf~\cite{BandtGraf_1992self}. It asserts that a self-similar IFS does \textbf{not} satisfy the OSC if and only if
\begin{equation*}
\forall  \varepsilon>0, \exists \  \iiv\neq \jjv\in\Sigma^\ast:\; \|S_{\iiv}^{-1}\circ S_{\jjv}-\mathrm{Id}\|<\varepsilon.
\end{equation*}
One can easily get that for an arbitrary $x,x_0\in\mathbb{R}$ we have
\begin{equation}\label{eq:79}
\big(S_{\iiv}^{-1}\circ S_{\jjv}-\mathrm{Id}\big)(x) = \left(\frac{r_{\jjv}}{r_{\iiv}} -1\right) (x-x_0)  + \frac{S_{\jjv}(x_0)-S_{\iiv}(x_0)}{r_{\iiv}}\,.
\end{equation}

Let $B \subset \mathbb{R}^d$ be a non-degenerate ball and let
\begin{equation*}\label{eq:22}
\mathcal{S}^{\underline{\lambda}} = \big\{S^{\underline{\lambda}}_i(x)= r_i(\underline{\lambda}) x +t_i(\underline{\lambda})\big\}_{i=1}^m, \;\;\text{where } r_i(\underline{\lambda} )\in(-1,1)\setminus\{0\},
\end{equation*}
be a transversal family of self-similar IFSs on the line as in Definition~\ref{def:97}. In particular, $0<\beta\leq |r_i(\underline{\lambda} )|\leq \rho<1$ for all $\underline{\lambda} \in B$ and $i\in[m]:=\{1,\ldots,m\}$. For an $\varepsilon>0,$ we write
\begin{multline*}
\mathcal{V}_{\varepsilon}:=
\Big\{
\underline{\lambda}\in B :
\text{ there exists } \iiv\ne\jjv\in\Sigma^{\ast} \text{ such that}\\
\frac{r_{\iiv}(\underline{\lambda})}
{ r_{\jjv}(\underline{\lambda}) }
\in (e^{-\varepsilon},e^{\varepsilon})
\;\mbox{ and }\;
|\Pi_{\underline{\lambda}}(\iiv\mathbf{1})-\Pi_{\underline{\lambda}}(\jjv\mathbf{1})| < \varepsilon\cdot r_{\jjv}(\underline{\lambda})
\Big\},
\end{multline*}
where $\mathbf{1}:=(1,1, \dots)$.
Then \eqref{eq:79} with the choice $x_0=\Pi_{\underline{\lambda}}(\mathbf{1})$ immediately implies that for every $\underline{\lambda} \in B$:
\begin{equation*}\label{eq:69}
\underline{\lambda}\not\in\mathscr{OSC}
 \;\Longleftrightarrow\; \underline{\lambda}\in \bigcap_{n>0} \mathcal{V}_{1/n}.
\end{equation*}
Hence, to prove the proposition it is enough to show that for every $\varepsilon>0$ the set $B\setminus (\mathscr{SSP}\cup\mathcal{V}_{\varepsilon})$ has no density point. We already argued in~\eqref{eq:503} that for every $\underline{\lambda}_0\notin\mathscr{SSP}$ and $\delta>0$ there exits $\underline{\lambda}_1\in B(\underline{\lambda}_0,\delta)$ such that $\underline{\lambda}_1\notin \mathscr{OSC}$. In particular, $\mathcal{V}_{\varepsilon}$ is dense in $B\setminus\mathscr{SSP}$. However, we need a more quantitative dependence between the parameters, see \eqref{eq:64}.

We start with a technical lemma.
For a $\underline{\lambda} _0\in B$ and a $k \geq 1$ we define the corresponding Moran class:
\begin{equation*}\label{eq:72}
\mathcal{M}_k(\underline{\lambda} _0):=
\big\{
\iiv\in\Sigma^{\ast}:
|r_{\iiv}(\underline{\lambda} _0)| \leq \rho^k
<  |r_{\iiv^-}(\underline{\lambda} _0)|
\big\},
\end{equation*}
where $r_{\iiv^-}(\underline{\lambda}_0):
=r_{i_1}(\underline{\lambda} _0)\cdots
r_{i_{|\iiv|-1}}(\underline{\lambda} _0).
$
The proof of the following Lemma is the combination of the proofs of
\cite[Lemma 3.2 and Lemma 3.3]{PSS_SelfSIm_2000IsrlJ}.
\begin{lemma}\label{lem:99}
	For every $\varepsilon>0$ and $\underline{\lambda}_0\in B$ we can find an $N=N(\varepsilon,\underline{\lambda} _0)$ such that
	for every $k \geq 1$ and
	$\iiv,\jjv\in\mathcal{M}_k(\underline{\lambda} _0)$
	there exists $\mathbf{u},\mathbf{v}\in \Sigma^{\ast}$
	such that
	\begin{description}
		\item[(a)] $\iiv$ is a prefix of $\mathbf{u}$ and
		$|\mathbf{u}|-|\iiv| \leq N$
		\item[(b)] $\jjv$ is a prefix of $\mathbf{v}$ and
		$|\mathbf{v}|-|\jjv| \leq N$
	\end{description}
	satisfying
	\begin{equation*}\label{eq:70}
	\frac{r_{\mathbf{u}}(\underline{\lambda} _0)}
	{r_{\mathbf{v}}(\underline{\lambda} _0)}
	\in\left(\e{-\varepsilon/3},\e{\varepsilon/3}\right).
	\end{equation*}
	Moreover, assume that
	$\underline{\lambda} \in B$
	satisfies
	\begin{equation*}\label{eq:73}
	\|\underline{\lambda} -\underline{\lambda} _0\|
	\leq
	\frac{\varepsilon \beta}{3L(k+N)},
	\end{equation*}
	where $L:=\max\limits_{i\in[m],\underline{\lambda} \in B}\| \nabla r_i(\underline{\lambda} )\|$.
	Then we have
	\begin{equation}\label{eq:71}
	\frac{r_{\mathbf{u}}(\underline{\lambda} )}
	{r_{\mathbf{u}}(\underline{\lambda}_0 )}
	\in
	\left(\e{-\varepsilon/3},\e{\varepsilon/3}\right)
	\mbox{ and  }
	\frac{r_{\mathbf{u}}(\underline{\lambda} )}
	{r_{\mathbf{v}}(\underline{\lambda} )}
	\in
	\left(\e{-\varepsilon},\e{\varepsilon}\right).
	\end{equation}
\end{lemma}
Let us introduce some notation,
$$
T:=\max\limits_{i\in[m],\underline{\lambda} \in B}
t_i(\underline{\lambda} ),
\quad
\widetilde{T}:=\max\limits_{i\in[m],\underline{\lambda} \in B}
\|\nabla t_i(\underline{\lambda} )\|,
%\quad
%\widetilde{R}:=\max\limits_{i\in[m],\underline{\lambda} \in B}
%\|\nabla r_i(\underline{\lambda} )\|,
$$
%where $[m]=\left\{1, \dots ,m\right\}$
and
\begin{equation*}
f_{\ii,\jj}(\underline{\lambda}) := \Pi_{\underline{\lambda}}(\ii)-\Pi_{\underline{\lambda}}(\jj),
\end{equation*}
where in the self-similar case $\Pi_{\underline{\lambda}}(\ii)=\sum_{n=1}^{\infty}r_{i_1\ldots {i_{n-1}}}t_{i_n}$. Hence, for all $\mathbf{i},\mathbf{j}\in\Sigma$,
\begin{equation}\label{eq:75}
\max\limits_{\underline{\lambda} \in B} |f_{\ii,\jj}(\underline{\lambda})| \leq \frac{2T|r_{\mathbf{i}\wedge\mathbf{j}}(\underline{\lambda} )|}{1-\rho},\quad
\max\limits_{\underline{\lambda} \in B} \|\nabla_{\underline{\lambda} }f_{\ii,\jj}(\underline{\lambda})\| \leq \frac{2\widetilde{T}}{1-\rho}
+
\frac{2TL}{(1-\rho)^2}:=C_2.
\end{equation}
Moreover, there exists a $K>0$ such that
\begin{equation}\label{eq:67}
\left|
\frac{\partial f_{\mathbf{i},\mathbf{j}}}{\partial\lambda_\ell }(\underline{\lambda} _1)
-
\frac{\partial f_{\mathbf{i},\mathbf{j}}}{\partial\lambda_\ell }(\underline{\lambda} _2)
\right|
<
K \cdot
\|\underline{\lambda} _1-\underline{\lambda} _2\|,
\quad
\forall \underline{\lambda} _1,\underline{\lambda} _2\in B,
\ell \in[m].
\end{equation}

%\begin{claim}\label{claim:01}
%For every $\varepsilon>0$
%\begin{enumerate}[(i)]
%\item $\mathcal{V}_{\varepsilon}$ is dense in $\mathscr{NOSSP}$,
%\item $\mathscr{NOSSP}\setminus \mathcal{V}_{\varepsilon}$ has no density point.
%\end{enumerate}
%\end{claim}

\subsection*{Proof of Proposition~\ref{claim:51}}

Let $\underline{\lambda}_0\in(B\setminus\mathscr{SSP})\cap(\mathrm{int}B)$. Hence, we can choose $\ii,\jj\in\Sigma$ with
\begin{equation*}
i_1\neq j_1\quad
\mbox{ and }\quad
f_{\ii,\jj}(\underline{\lambda}_0) = 0.
\end{equation*}
Let
$$
\delta:=\frac{\zeta}{2K\sqrt{m}},
$$
where $K$ was defined in \eqref{eq:67}. 
Choose $k$ large enough   satisfying
\begin{equation}\label{cv80}
  \frac{4T}{1-\rho}\rho^k<\zeta\frac{\min\{\delta,\mathrm{dist}(\lambda_0,\partial B)\}}{4\sqrt{m}}.
\end{equation}
Then we choose $n$ and $p$ such that 
\begin{equation}\label{eq:04}
\mathbf{i}|_n,\mathbf{j}|_p\in\mathcal{M}_k.
\end{equation}
Now we fix an $\varepsilon>0$  and apply Lemma~\ref{lem:99} for $\iiv:=\mathbf{i}|_n$ and $\jjv:=\mathbf{j}|_p$.
Then Lemma~\ref{lem:99} defines us $\underline{u},\underline{v}\in\Sigma^{\ast}$
and a constant $N$ independent of $\ii,\jj$ such that
\begin{align*}
&\underline{u}|_n = \ii|_n \text{ and } \underline{v}|_p = \jj|_p,  \\
&|\underline{u}|\leq n+N \text{ and } |\underline{v}|\leq p+N, \\
&\frac{|r_{\underline{u}}(\underline{\lambda} _0)|}{|r_{\underline{v}}(\underline{\lambda} _0)|} \in (e^{-\varepsilon/3},e^{\varepsilon/3}).
\end{align*}
For brevity, let $f(\underline{\lambda}):= f_{\underline{u}\mathbf{1},\underline{v}\mathbf{1}}(\underline{\lambda})$. The choice of $\underline{u}$ and $\underline{v}$ implies that
\begin{equation*}\label{eq:68}
| f(\underline{\lambda}_0 )|
\leq
|f_{\underline{u}\mathbf{1},\mathbf{i}}(\underline{\lambda} _0)|
+
|f_{\mathbf{i},\mathbf{j}}(\underline{\lambda} _0)|+
|f_{\underline{v}\mathbf{1},\mathbf{j}}(\underline{\lambda} _0)|
\leq \frac{4T}{1-\rho}\rho^k<\zeta,
\end{equation*}
where the last inequality holds by \eqref{cv80}.Then by the transversality condition \eqref{eq:15} there exists an $\ell \in[m]$ such that
$\frac{\partial f}{\partial\lambda_{\ell }}(\underline{\lambda} _0)>\frac{\zeta}{\sqrt{m}}$.
Using that
$$
\left|
\frac{\partial f}{\partial \lambda_{\ell} }(\underline{\lambda} )
\right| \geq K\|\underline{\lambda} -\underline{\lambda} _0\|+
\frac{\partial f}{\partial \lambda_{\ell} }(\underline{\lambda}_0 ),
$$
we obtain that
\begin{equation}\label{eq:66}
\exists \ell \in[m] \mbox{ such that }  \|\underline{\lambda} -\underline{\lambda} _0\|<\delta
\Longrightarrow
\left|
\frac{\partial f}{\partial \lambda_{\ell} }(\underline{\lambda} )
\right|>\frac{\zeta}{2\sqrt{m}}.
\end{equation}
So, if we choose $k$ so large that
\begin{equation}\label{eq:65}
|f(\underline{\lambda}_0 )|<  \frac{4T}{1-\rho}\rho^k<\frac{\delta\zeta}{4\sqrt{m}},
\end{equation}
then combining \eqref{eq:66} and \eqref{eq:65}, we get that there is $\underline{\lambda} _1$
such that
\begin{equation}\label{eq:64}
f(\underline{\lambda} _1)=0
\mbox{ and }
\|\underline{\lambda}_0 -\underline{\lambda} _1\|
<
\frac{|f(\underline{\lambda} _0)|}{\min\limits_{\|\underline{\lambda} -\underline{\lambda} _0\|<\delta} \left|
	\frac{\partial f}{\partial \lambda_{\ell} }(\underline{\lambda} )
	\right|}
<
\frac{8T\sqrt{m}}{\zeta \cdot (1-\rho)}\rho^{k}
<
\frac{\delta}{2}.
\end{equation}
By the definition of $k$ we 
also see that $\underline{\lambda} _1\in (\mathrm{int} B)$.  We write
\begin{equation*}\label{eq:63}
\eta_k:=\frac{8T\sqrt{m}}{\zeta \cdot (1-\rho)}\rho^{k},\qquad
F_k:=B\left(
\underline{\lambda} _0,2 \cdot \eta_k
\right).
\end{equation*}
Recall that we defined $C_2$ in \eqref{eq:75}.
Let
$$
\xi:=\xi_{k,\varepsilon}:=\frac{\beta^{N+1}\rho^k\varepsilon}{2C_2}.
$$
We claim that $B\left(\underline{\lambda} _1,\xi_{k,\varepsilon}\right)
\subset \mathcal{V}_\varepsilon$ with $\underline{u}\neq\underline{v}\in\Sigma^\ast$, i.e. for every $\underline{\lambda} \in B\left(\underline{\lambda} _1,\xi_{k,\varepsilon}\right)$
\begin{equation}\label{eq:61}
\frac{r_{\underline{u}}(\underline{\lambda} )}{r_{\underline{v}}(\underline{\lambda} )}\in\left(\e{-\varepsilon},\e{\varepsilon}\right)
\end{equation}
and
\begin{equation}\label{eq:58}
|f_{\underline{u}\mathbf{1},\underline{v}\mathbf{1}}
(\underline{\lambda})| < \varepsilon\cdot r_{\underline{u}}.
\end{equation}

First of all, we choose $\varepsilon$ small enough such that
$\xi_{k,\varepsilon}<\eta_k$ for every $k$.
Note that by this and \eqref{eq:64} we have $B\left(\underline{\lambda} _1,\xi_{k,\varepsilon}\right) \subset
F_k$. Moreover, if we choose $k $ so large that
\begin{equation}\label{eq:60}
\eta_k<\frac{\varepsilon\beta}{3L(k+N)},
\end{equation}
then Lemma~\ref{lem:99} implies \eqref{eq:61} for all $\underline{\lambda} \in F_k\supset B\left(\underline{\lambda} _1,\xi_{k,\varepsilon}\right)$.

%Now we prove that
%\begin{equation}\label{eq:62}
%B\left(\underline{\lambda} _1,\xi_{k,\varepsilon}\right)
%\subset \mathcal{V}_\varepsilon.
%\end{equation}

%To verify \eqref{eq:62} it remains to prove that
%\begin{equation}\label{eq:58}
%|f_{\underline{u}\mathbf{1},\underline{v}\mathbf{1}}
%(\underline{\lambda})| < \varepsilon\cdot r_{\underline{u}} \mbox{ for }
%\underline{\lambda} \in B\left(\underline{\lambda} _1,\xi_{k,\varepsilon}\right).
%\end{equation}

To show~\eqref{eq:58}, we fix an arbitrary
$
\underline{\lambda} \in B\left(\underline{\lambda} _1,\xi_{k,\varepsilon}\right)
$.
First we observe that
\begin{equation}\label{eq:57}
\beta^{N+1}\rho^k \leq r_{\underline{u}}(\underline{\lambda} _0)
\leq
2r_{\underline{u}}(\underline{\lambda} ).
\end{equation}
Indeed, the left hand-side in \eqref{eq:57} holds
since $\mathbf{i}|_n\in\mathcal{M}_k$ and $|\underline{u}|-n<N$. The right hand-side of \eqref{eq:57}
follows  from the fact that we choose
$\varepsilon>0$ so small that $\e{\varepsilon/3}<2$ and
\eqref{eq:60} implies that
$\xi_{k,\varepsilon}<\eta_k<\varepsilon\beta/(3L(k+N))$, so \eqref{eq:71} holds, which implies the right hand-side of
\eqref{eq:57}. By \eqref{eq:57}, to show
\eqref{eq:58}, we only need to prove that
\begin{equation}\label{eq:56}
|f_{\underline{u}\mathbf{1},\underline{v}\mathbf{1}}
(\underline{\lambda})| <
\frac{\varepsilon}{2}
\beta^{N+1}\rho^k \;\mbox{ for }
\underline{\lambda} \in  B\left(\underline{\lambda} _1,\xi_{k,\varepsilon}\right).
\end{equation}
To see this, recall the definition of $C_2$ which was given in \eqref{eq:75}.
By the mean-value inequality we get
\begin{equation*}
|f_{\underline{u}\mathbf{1},\underline{v}\mathbf{1}}
(\underline{\lambda})|
=
|f_{\underline{u}\mathbf{1},\underline{v}\mathbf{1}}
(\underline{\lambda})
-
f_{\underline{u}\mathbf{1},\underline{v}\mathbf{1}}
(\underline{\lambda}_1)
|
\leq
C_2
\|\underline{\lambda} -\underline{\lambda} _1\|
\leq
\frac{\varepsilon}{2}\beta^{N+1}\rho^k.
\end{equation*}
That is \eqref{eq:56} holds and so \eqref{eq:58} holds. Thus, we showed that $B\left(\underline{\lambda} _1,\xi_{k,\varepsilon}\right)
\subset \mathcal{V}_\varepsilon$.

To conclude, observe that
\begin{equation*}\label{eq:55}
\frac{\xi_{k,\varepsilon}}{2\eta_k}=
\varepsilon \cdot
\frac{\beta^{N+1}\zeta(1-\rho)}{16 T \sqrt{m}C_2}.
\end{equation*}
This and $B\left(\underline{\lambda} _1,\xi_{k,\varepsilon}\right)
\subset \mathcal{V}_\varepsilon$ implies that for every $k$ large enough we have:
\begin{equation*}\label{eq:54}
\frac{\mathcal{L}_d\left(\left(B\setminus(\mathscr{SSP}\cup
	\mathcal{V}_\varepsilon)
	\right)^c\cap B(\underline{\lambda} _0,\eta_k)\right)}{\mathcal{L}_d\left(B(\underline{\lambda} _0,\eta_k)\right)}
>\mathrm{const} \cdot\varepsilon^m,
\end{equation*}
in other words $B\setminus(\mathscr{SSP}\cup
\mathcal{V}_\varepsilon)$ has no Lebesgue density point. This concludes the proof of Proposition~\ref{claim:51}.
	
\subsection*{Acknowledgment} BB was supported by the grants OTKA PD123970 and the J\'anos Bolyai Research Scholarship of the Hungarian Academy of Sciences. BB and SK were jointly supported by the grant OTKA K123782. IK was financially supported by a Leverhulme Trust Research Project Grant (RPG-2019-034). MR was supported by the National Science Centre grant 2019/33/B/ST1/00275 (Poland).

\bibliographystyle{abbrv}

\end{document}